\documentclass[a4paper]{article}

\title{A presentation for the pure Hilden group}
\author{Stephen Tawn}

\usepackage{amsmath, amssymb, amsthm, amsfonts}
\usepackage[colorlinks=true,linkcolor=blue]{hyperref} 
\usepackage{longtable}
\usepackage{array}
\usepackage{overpic}
\usepackage{enumerate}

\def\today{\number\day \space\ifcase\month\or
  January\or February\or March\or April\or May\or June\or
  July\or August\or September\or October\or November\or December\fi
  \space\number\year}

\newcommand{\Z}{\mathbb{Z}}
\newcommand{\R}{\mathbb{R}}

\newcommand{\st}{\boldsymbol{\mid}}
\newcommand{\boundary}{\partial}
\newcommand{\for}{\text{for }}
\newcommand{\union}{\cup}
\newcommand{\Union}{\bigcup}
\newcommand{\intersect}{\cap}
\newcommand{\includes}{\hookrightarrow}

\DeclareMathOperator{\stab}{Stab}
\DeclareMathOperator{\aut}{Aut}

\newcommand{\Braid}[1]{\mathbf{B}_{#1}}
\newcommand{\PureBraid}[1]{\mathbf{P}_{#1}}
\newcommand{\FramedBraid}[1]{\mathbf{F\!B}_{#1}}
\newcommand{\Hilden}[1]{\mathbf{H}_{#1}}
\newcommand{\PureHilden}[1]{\mathbf{P\!H}_{#1}}
\newcommand{\FramedPureBraid}[1]{\mathbf{F\!P}_{#1}}
\newcommand{\MCG}{\mathbf{MCG}}
\newcommand{\X}[1]{\mathbf{X}_{#1}}
\newcommand{\halfspace}{\R^3_+}

\newcommand{\deductionTableWidth}{ xxxxxxxxxxxxxx & xxxxxxxxxxxxxxxxxxxxxx & xxxxxxxxxx \kill}

\newtheorem{theorem}{Theorem}
\newtheorem{lemma}[theorem]{Lemma}
\newtheorem{proposition}[theorem]{Proposition}
\newtheorem{claim}{Claim}

\theoremstyle{definition}
\newtheorem{definition}[theorem]{Definition}

\newcolumntype{L}{>{\(}l<{\)}}
\newcolumntype{C}{>{\(}c<{\)}}
\newcolumntype{R}{>{\(}r<{\)}}

\newcommand{\caps}{
  \begin{overpic}[width=9cm,height=2.4cm]{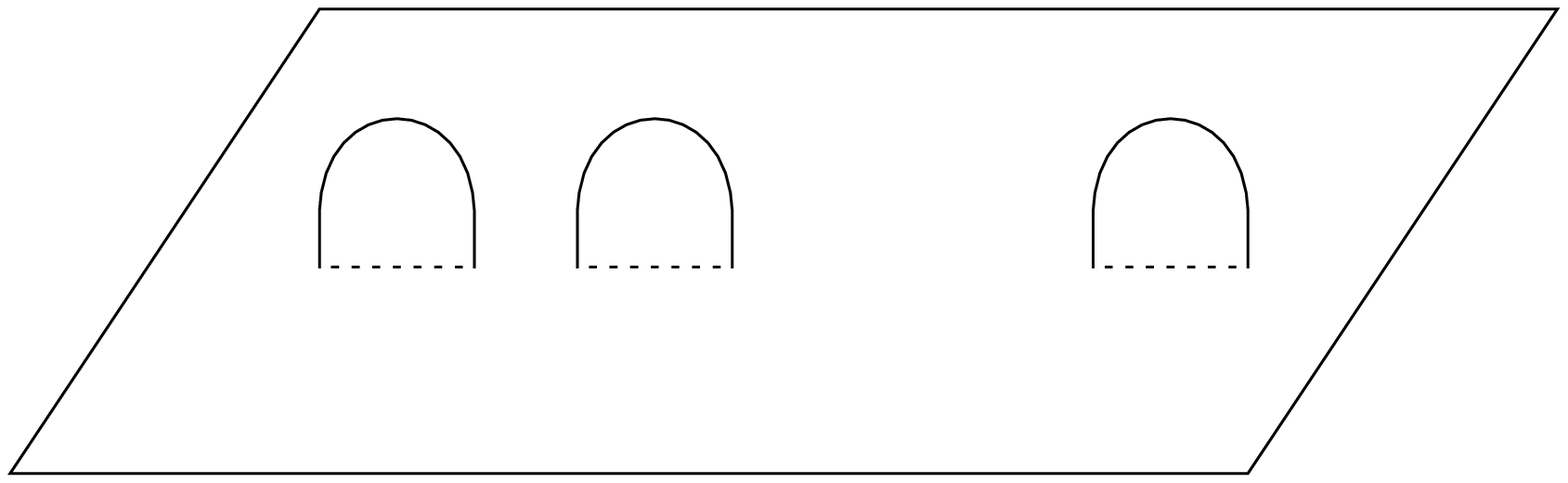}
    \put(56.5,16){$\ldots$}
    \put(29,20){$a_1$}
    \put(45.5,20){$a_2$}
    \put(79,20){$a_n$}
    \put(23.5,15){$d_1$}
    \put(40.4,15){$d_2$}
    \put(73.5,15){$d_n$}
  \end{overpic}
}

\newcommand{\generatorp}{
  \parbox[c]{2cm}{
    \begin{overpic}[width=1.7cm]{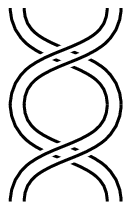}
      \put(4,100){$i$}
      \put(51,100){$j$}
    \end{overpic}
  }
}

\newcommand{\generatorx}{
  \parbox[c]{2cm}{
    \begin{overpic}[width=1.7cm]{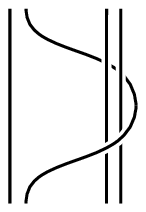}
      \put(4,100){$i$}
      \put(51,100){$j$}
    \end{overpic}
  }
}

\newcommand{\generatory}{
  \parbox[c]{2cm}{
    \begin{overpic}[width=1.7cm]{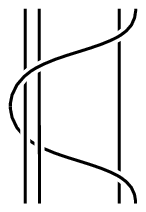}
      \put(12,100){$i$}
      \put(58,100){$j$}
    \end{overpic}
  }
}

\newcommand{\generatort}{
  \parbox[c]{0.45cm}{
    \begin{overpic}[width=0.45cm]{images/t}
      \put(5,98){$k$}
    \end{overpic}
  }
}

\newcommand{\edgevovoXij}{
  \fbox{\parbox[c]{3cm}{
    \begin{overpic}[width=3cm]{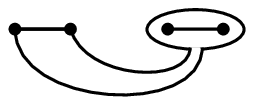}
      \put(25,37){$i$}
      \put(73,37){$j$}
    \end{overpic}}}
}

\newcommand{\edgevovoxij}{
  \fbox{\parbox[c]{3cm}{
    \begin{overpic}[width=3cm]{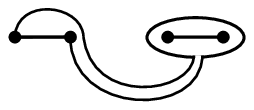}
      \put(25,37){$i$}
      \put(73,37){$j$}
    \end{overpic}}}
}

\newcommand{\edgevovoYij}{
  \fbox{\parbox[c]{3cm}{
    \begin{overpic}[width=3cm]{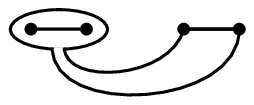}
      \put(24,37){$j$}
      \put(74,37){$i$}
    \end{overpic}}}
}

\newcommand{\edgevovoyij}{
  \fbox{\parbox[c]{3cm}{
    \begin{overpic}[width=3cm]{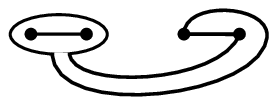}
      \put(24,37){$j$}
      \put(74,37){$i$}
    \end{overpic}}}
}

\newcommand{\edgevovoXikXij}{
  \fbox{\parbox[c]{3.8cm}{
    \begin{overpic}[width=3.8cm]{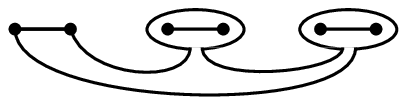}
      \put(17,24){$i$}
      \put(48,24){$j$}
      \put(81,24){$k$}
    \end{overpic}}}
}

\newcommand{\edgevovoxijxik}{
  \fbox{\parbox[c]{3.8cm}{
    \begin{overpic}[width=3.8cm]{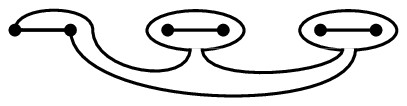}
      \put(17,24){$i$}
      \put(48,24){$j$}
      \put(81,24){$k$}
    \end{overpic}
    }
  }
}

\newcommand{\edgevovoxjkyij}{
  \fbox{\parbox[c]{3.8cm}{
    \begin{overpic}[width=3.8cm]{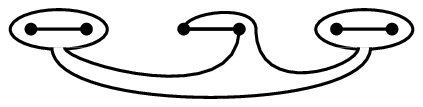}
      \put(16,24){$j$}
      \put(49,24){$i$}
      \put(81,24){$k$}
    \end{overpic}
    }
  }
}

\newcommand{\edgevovoYijXjk}{
  \fbox{\parbox[c]{3.8cm}{
    \begin{overpic}[width=3.8cm]{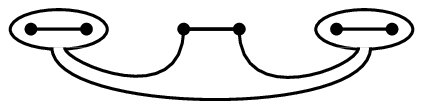}
      \put(16,24){$j$}
      \put(49,24){$i$}
      \put(81,24){$k$}
    \end{overpic}}}
}

\newcommand{\edgevovoYjkYik}{
  \fbox{\parbox[c]{3.8cm}{
    \begin{overpic}[width=3.8cm]{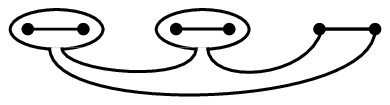}
      \put(17,24){$j$}
      \put(49,24){$k$}
      \put(81,24){$i$}
    \end{overpic}}}
}

\newcommand{\edgevovoyikyjk}{
  \fbox{\parbox[c]{3.8cm}{
    \begin{overpic}[width=3.8cm]{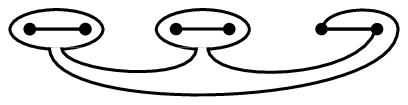}
      \put(17,24){$j$}
      \put(49,24){$k$}
      \put(81,24){$i$}
    \end{overpic}
    }
  }
}

\newcommand{\generatorsigma}{
  \parbox[c]{2.5cm}{
    \begin{overpic}[height=2cm]{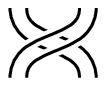}
      \put(10,83){$i$}
      \put(68,83){$i+1$}
    \end{overpic}
  }
}

\newcommand{\generatortau}{
  \parbox[c]{1cm}{
    \begin{overpic}[height=2cm]{images/tau}
      \put(15,100){$j$}
    \end{overpic}
  }
}

\begin{document}

\maketitle

\begin{abstract}
Consider the unit ball, $B = D \times [0,1]$, containing $n$ unknotted
arcs $a_1, a_2, \ldots, a_n$ such that the boundary of each $a_i$ lies
in $D \times \{0\}$. The Hilden (or Wicket) group is the mapping class
group of $B$ fixing the arcs $a_1 \union a_2 \union \ldots \union a_n$
setwise and fixing $D \times \{1\}$ pointwise.  This group can be
considered as a subgroup of the braid group.  The pure Hilden group is
defined to be the intersection of the Hilden group and the pure braid
group.

In a previous paper we computed a presentaion for the Hilden group
using an action of the group on a cellular complex.  This paper uses
the same action and complex to calculate a finite presentation for the
pure Hilden group.  The framed braid group acts on the pure Hilden group
by conjugation and this action is used to reduce the number of cases.
\end{abstract}

\section{Introduction}

Given a braid $b \in \Braid{2n}$ on $2n$ strings we can produce
a link by taking its plat closure.  This is formed by adding
semi-sircular caps and cups connecting consecutive pairs of
strings at the top and at the bottom.

\begin{figure}[!htb]
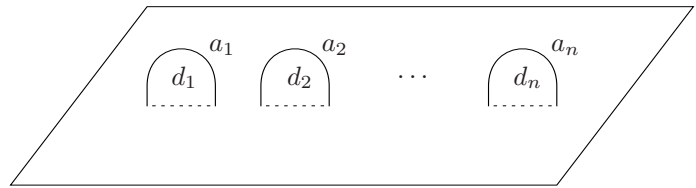

  \[ \caps \]
  \caption{The caps $a_i$ and discs $d_i$} \label{caps}
\end{figure}

Let $a = a_1 \union a_2 \union \cdots \union a_n$ be the
$(0,2n)$--tangle given by the caps.  The Hilden (or wicket)
subgroup of the braid group is the stabiliser of $a$ under
the action of the braid group on the set of $(0,2n)$--tangles.
\[
    \Hilden{2n} = \{ b \in \Braid{2n} \st a\,b = a \}
\]
We define the pure Hilden group to be the intersection of the
Hilden group and the pure braid group.
\[
    \PureHilden{2n} = \PureBraid{2n} \intersect \Hilden{2n}
\]

There are two moves that can be perfomed on a braid $b \in \Braid{2n}$
which leave its plat closure unchanged. A double coset move where
you multiply on the left and right by elements of the Hilden
group and a stabilisation move where you add two extra strings
on the right and then multiply by $\sigma_{2n}$.  Birman\cite{Birman}
has shown that any two braids with isotopic plat closures can
be related by a sequence of these double coset and stabilisation moves.

Generators for the equivalent subgroup of the braid group of the sphere
were found by Hilden\cite{Hilden} and a finite presentation for the
Hilden group was calculated independently by the author\cite{Tawn} and
Brendle--Hatcher\cite{BrendleHatcher}.

If we shift the cups so that the first string is connected to the last,
the second to the third, etc., then we get a modified form of plat
closure (or short-circuit map) which takes pure braids to knots.  Now
the stabiliser of the cups is different to that of the caps and we can
use inclusion for the stabilisation move.  Mostovoy--Stanford\cite{MostovoyStanford}
show that if you take the limit of this system of inclusions then
modified plat closure induces a bijection between 
$\PureHilden{\infty}^{\text{top}} \backslash \PureBraid{\infty}
                                / \PureHilden{\infty}^{\text{bottom}}$
and the set of oriented links.

In this paper we will compute a finite presentation for the
pure Hilden group $\PureHilden{2n}$.

\begin{theorem} \label{main-theorem}
  The pure Hilden group has a finite presentation with generating
  set $S$ and relations $R$
  \[
     \PureHilden{2n} = \langle S \mid R \rangle
  \]
  where $S$ and $R$ are as follows.

  Let
  \[
     S = \left\{ p_{ij},\ x_{ij},\ y_{ij},\ t_k 
             \st 1 \leq i < j \leq n, 1 \leq k \leq n \right\}
  \]
  where $p_{ij}=p_{ji}$, $x_{ij}=x_{ji}$, $y_{ij}=y_{ji}$ and
  $t_k$ are the following elements of $\PureHilden{2n}$.  Here
  all of the remaining strings lie behind those shown.
  \[
     p_{ij} = \generatorp \qquad
     x_{ij} = \generatorx \qquad
     y_{ij} = \generatory \qquad
     t_k    = \generatort 
  \]

  Let $R$ be the following relations. 
\begin{align}
  p_{ij}\,t_{k} & = t_k\,p_{ij}  &&      \tag{C-$pt$} \label{C-pt} \\
     t_i\,t_{j} & = t_j\,t_i     &&      \tag{C-$tt$} \label{C-tt} \\
  x_{ij}\,t_{k} & = t_k\,x_{ij}  && i < j \qquad k \neq i
                                         \tag{C-$xt$} \label{C-xt} \\
  y_{ij}\,t_{k} & = t_k\,y_{ij}  && i < j \qquad k \neq j
                                         \tag{C-$yt$} \label{C-yt} \\
  \alpha_{ij}\,\beta_{kl} & = \beta_{kl}\,\alpha_{ij}
          && \begin{array}{c}
                    \alpha, \beta \in \{p, x, y\}, \\
                    (i, j, k, l) \text{ cyclically ordered}
             \end{array} \tag{C1} \label{C1} \\
  \alpha_{ij}\,\,\beta_{ik}\,\gamma_{jk}
                & = \beta_{ik}\,\gamma_{jk}\,\,\alpha_{ij}
          && \begin{array}{c}
                    (i, j, k) \text{ cyclically ordered,} \\
                    (\alpha, \beta, \gamma)
                    \text{ as in Table~\ref{C2-table}}
             \end{array} \tag{C2} \label{C2} \\
  \alpha_{ik}\,\,p_{jk}\,\beta_{jl}\,p_{jk}^{-1}
      & = p_{jk}\,\beta_{jl}\,p_{jk}^{-1}\,\,\alpha_{ik}
          && \begin{array}{c}
                    \alpha, \beta \in \{p, x, y\}, \\
                    (i, j, k, l) \text{ cyclically ordered} 
             \end{array} \tag{C3} \label{C3} \\
  x_{ij}\,p_{ij}\,t_i &= p_{ij}\,t_i\,x_{ij} && i < j
                                         \tag{M-$x$} \label{M-x} \\
  y_{ij}\,p_{ij}\,t_j &= p_{ij}\,t_j\,y_{ij} && i < j
                                         \tag{M-$y$} \label{M-y} 
\end{align}

\begin{table}[!htbp]
  \centering
  \begin{tabular}{C|C}
     i < j < k & \begin{array}{cccc} (p,p,p) & (p,y,y) & (x,p,p) & (x,x,p) \\
                                     (x,y,y) & (y,p,p) & (y,p,x) & (y,y,y) 
                 \end{array} \\ \hline 
     j < k < i & \begin{array}{cccc}(p,p,p) & (p,x,y) & (x,p,p) & (x,p,x) \\
                                    (x,x,y) & (y,p,p) & (y,x,y) & (y,y,p) 
                 \end{array} \\ \hline
     k < i < j & \begin{array}{cccc} (p,p,p) & (p,x,x) & (x,p,p) & (x,x,x) \\
                                     (x,y,p) & (y,p,p) & (y,p,y) & (y,x,x) 
                 \end{array}
  \end{tabular}
  \caption{The values of $(\alpha, \beta, \gamma)$ for \eqref{C2}\protect\footnotemark}
  \label{C2-table}
\end{table}

\footnotetext{In fact Table~\ref{C2-table} lists all possible triples for which
    \eqref{C2} holds.  These were found using the {\sc Magma} computational
     algebra system\cite{Magma}.}

\end{theorem}

As with the braid group, the Hilden group can be viewed as a mapping
class group.  Let $B^3_+$ be a half ball such that it contains the caps
and let $S^2_+ = \boundary B^3_+$ be its boundary.  The half ball and 
half sphere intersect the plane in a 2-ball $B^2$ and a circle $S^1$.
We now have that $\Hilden{2n} = \MCG(B^3_+, a, S^2_+)$, i.e.\ the group of
isotopy classes of self homeomorphisms of $B^3_+$ which preserve $a$
setwise and $S^2_+$ pointwise.  The inclusion $(B^2, \boundary a, S^1)
\includes (B^3_+, a, S^2_+)$ induces the embedding $\Hilden{2n} \includes
\Braid{2n}$.

In \cite{Tawn} we used the mapping class viewpoint to define an action
of the Hilden group on a cellular complex. We then used the method
of Hatcher--Thurston\cite{HatcherThurston}, Wajnryb\cite{Wajnryb1}%
\cite{Wajnryb2}\cite{Wajnryb3}, etc.\ to compute a presentation from this
action.  In this paper we will use the same method with the same complex
and action to compute a presentation for the pure Hilden group.

We recall the method in Section~\ref{method}, the complex in
Section~\ref{complex} and go on to compute the vertex stabiliser and
edge orbits in Section~\ref{stabiliser} and Section~\ref{edges}.
To reduce the number of cases we will use an action of the framed braid
group on the pure Hilden group.  The required properties of this action
are given in Section~\ref{action-properties}.  In Sections~\ref{R1},
\ref{R2} and \ref{R3}, we make use of this action to show that the $R_1$,
$R_2$ and $R_3$ relations follow from $R$.  We then finish by
constructing this action and showing that it satisfies the required
properties in Section~\ref{definePhi}.

\section{The method} \label{method}

We will now summarise \S 2 of \cite{Tawn} which in turn follows \S 2 
``Une M\'ethode pour pr\'esenter G'' of Laudenbach\cite{Laudenbach}.
This is the method used by Hatcher--Thurston\cite{HatcherThurston}, 
Wajnryb\cite{Wajnryb1}\cite{Wajnryb2}\cite{Wajnryb3}, etc.\ to calculate
presentations for surface and handlebody mapping class groups.  

Suppose that $X$ is a connected simply-connected cellular 2-complex
such that each attaching map is injective and that each cell is uniquely
determined by its boundary.  Suppose that $G$ is a group acting
cellularly on the right of $X$, and that this action is transitive on
the vertex set $X^0$. Pick a vertex $v_0 \in X^0$ as a basepoint and
let $H$ denote its stabiliser in $G$, i.e.\ $H = \{g \in G \st v_0 \cdot
g = v_0 \}$.  Suppose that $H$ has a presentation with generating set
$S_0$ and relations $R_0$, i.e.\ $H = \langle S_0 | R_0 \rangle$.

Given vertices $u,v \in X^0$ such that $\{u,v\}$ is the boundary of an
edge of $X$ we will write $(u,v)$ for this (oriented) edge.
Given a sequence $v_1, v_2, \ldots, v_k$ of vertices such that either
$v_i = v_{i+1}$ or $(v_i, v_{i+1})$ forms an edge we will write
$(v_1, v_2, \ldots, v_k)$ for the path traversing these edges.
Whenever $v_i = v_{i+1}$ we shall say that $v_i$ is a stationary point.

Suppose that $\{e_\lambda\}_{\lambda\in\Lambda}$ is a set of representatives
for the orbits of the edges of $X$, i.e.
$X^1 = \Union_{\lambda \in \Lambda} e_\lambda G$ and
$e_\lambda G = e_{\lambda'} G $ only if $\lambda = \lambda'$.
Since the action of $G$ is transitive on $X^0$ we may assume that each
$e_\lambda$ starts at $v_0$ and that we can find $r_\lambda \in G$
such that each $e_\lambda = (v_0, v_0 \cdot r_\lambda)$.  Let
$S_1 = \{r_\lambda\}_{\lambda\in\Lambda}$.

Suppose that $\{f_\mu\}_{\mu\in M}$ is a set of representatives for the
orbits of the faces of $X$.  Again, since the action is transitive on
$X^0$, we may assume that the boundary of each face $f_\mu$ contains
the vertex $v_0$.

\begin{definition}
  An \emph{h-product of length} $k$ is a word of the form
  \[
    h_{k+1}\ r_{\lambda_k} h_k\ r_{\lambda_{k-1}} h_{k-1}\ 
                                             \cdots\ r_{\lambda_1} h_1
  \]
  where each $\lambda_i \in \Lambda$ and each of the $h_i$ are words in
  $H$. To each h-product we can associate an edge path
  $P = (v_0, v_1, \ldots, v_k)$ in $X$ starting at $v_0$ then visiting
  the vertices $v_1 = v_0 \cdot r_{\lambda_1} h_1$,
  $v_2 = v_0 \cdot r_{\lambda_2} h_2\ r_{\lambda_1} h_1$, etc. This means
  that the edge $(v_{i-1}, v_i)$ is in the orbit of
  $(v_0, v_0 \cdot r_{\lambda_i})$.  Given any edge path starting at
  $v_0$ we can choose an h-product to represent it.
\end{definition}

We can now choose the following three sets of relations.
\begin{itemize}
\item[$R_1$:] For each edge orbit representative $e_\lambda$ pick a
  generating set $T$ for the stabiliser of this edge, i.e.
  $\langle T \rangle = \stab_G(v_0)
                             \intersect \stab_G(v_0 \cdot r_\lambda)$.
  For each $t \in T$ we have the relation
  $r_{\lambda} t r_{\lambda}^{-1}=h$ for some word $h \in H$.

\item[$R_2$:] For each $e_\lambda$ we have a relation
  $r_{\lambda'} h\ r_\lambda = h'$ where the LHS is a choice of h-product
  for the path $(v_0,v_0 \cdot r_\lambda, v_0)$ and $h'$ is some word
  in $H$.

\item[$R_3$:]  For each face orbit representative $f_\mu$ with boundary
  $(v_0, v_1, \ldots, v_{k-1}, v_0)$ choose an h-product representing
  this path and a word $h \in H$ such that
  $r_{\lambda_k} h_k \cdots r_{\lambda_1} h_1 = h$.
\end{itemize}

\begin{theorem}\label{methodthm}
  The group $G$ has a presentation with generators $S_0$ and $S_1$ and
  relation $R_0$, $R_1$, $R_2$ and $R_3$.
  \[ G = \langle S_0 \union S_1
                     | R_0 \union R_1 \union R_2 \union R_3 \rangle \]
\end{theorem}

\section{The complex} \label{complex}

An embedded disc $D \subseteq \halfspace$ is said to \emph{cut out $a_i$}
if the interior of $D$ is disjoint from $a$, the arc $a_i$ is contained
in the boundary of $D$ and the boundary of $D$ lies in
$a_i \union \boundary \halfspace$, i.e.\ $a_i \subset \boundary D$ and
$\boundary D \subset a_i \union \boundary\halfspace$.
A \emph{cut system for $a$} is the isotopy class of $n$ pairwise
disjoint discs $\langle D_1,D_2,\ldots D_n\rangle$ where each $D_i$
cuts out the arc $a_i$. Say that two cut systems
$\langle D_1, D_2, \ldots, D_n \rangle$ and
$\langle E_1, E_2, \ldots, E_n \rangle$ differ by a simple move of length
$l$ if for some $i$ we have that $D_i \intersect E_i = a_i$, for all 
$j \neq i$ $D_j = E_j$ and the number of $a_i$ in the bounded component
of $\halfspace \setminus D_i \union E_i$ equals $l$.  If
this is the case we will suppress the non-changing discs and write
$\langle D_i \rangle \text{---} \langle E_i \rangle$.

We will say that a rectangle $(\langle D,E\rangle, \langle D',E\rangle,
\langle D',E'\rangle, \langle D,E'\rangle, \langle D,E\rangle)$ is
\emph{nested} if $E\union E'$ lies in the bounded component of
$\halfspace \setminus D \union D'$ or vice versa, i.e.\ if one pair of
changing discs lies underneath the other.

\begin{definition}
  Define the complex $\X{n}$ as follows.  The set of all
  cut systems for $a$ forms the vertex set $\X{n}^0$.  Two vertices
  are connected by a single edge iff they differ by a simple move of
  length one or two.  Finally, glue faces into every non-nested 
  rectangle of length one edges, every nested rectangle and every 
  triangle.
  Define the basepoint $v_0$ to be $\langle d_1, d_2, \ldots, d_n \rangle$
  where the $d_i$ are vertical discs below the $a_i$, see Figure~\ref{caps}.
\end{definition}

We will say that $a_j$ lies under the edge 
$(\langle D_i \rangle, \langle E_i \rangle)$ if it is contained in the
bounded component of the complement of $D_i \union E_i$.  At most two
discs lie under an edge.

In \cite{Tawn} we proved the following.
\begin{theorem} \label{X-simply-connected}
  The complex $\X{n}$ is connected and simply connected. \qed
\end{theorem}

Up to homotopy the group $\Hilden{2n}$ acts on $(\halfspace,a)$ by
homeomorphisms, therefore it takes cut systems to cut systems.  The
edges and faces of $\X{n}$ are determined by the intersections of
pairs of discs, hence this action on $\X{n}^0$ extends to a cellular
action on $\X{n}$.

\begin{theorem} \label{transitive}
  The action of $\PureHilden{2n}$ on $\X{n}^0$ is transitive.
\end{theorem}

\begin{proof}
  This exactly the same as the proof that the action of $\Hilden{2n}$
  on $\X{n}^0$ is transitive given in \cite{Tawn}.  All that is needed
  is to note that the constructed braids are pure. 

  Given a vertex $\langle D_1, D_2, \ldots, D_n \rangle$ of $\X{n}$,
  if we take each $i$ in turn and look at the intersection of $D_i$
  with $\R^2$.  We see that this defines a path from one end of $a_i$
  to the other.  If we now move one end around this path until it is
  close to the other and then move it straight back to its starting
  point we have an element of $\PureHilden{2n}$ that moves $D_i$ to
  $d_i$.  Combining all of these we see that $\langle D_1, D_2, \ldots,
  D_n \rangle$ is in the orbit of $v_0$, i.e.\ the action is transitive on
  $\X{n}^0$.
\end{proof}

\section{Vertex stabiliser} \label{stabiliser}

\begin{proposition} \label{vertex-stabiliser}
  The stabiliser of the vertex $v_0$ is the framed pure braid
  group $\FramedPureBraid{n}$ and so is isomorphic to 
  $\PureBraid{n} \times \Z^n$.
\end{proposition}

\begin{proof}
  If we restrict our attention to $\R^2$, elements of $\PureHilden{2n}$
  can be thought of as motions of the end points of the $a_i$.  For
  elements of the stabiliser of $v_0$ this motion moves the line segments
  $d_i \intersect \R^2$ so this is the fundamental group of configurations
  of $n$ ordered line segments in the plain, the framed pure braid group.
\end{proof}

The pure braid group has a presentation with generators $p_{ij}$ and
relations \eqref{C1}, \eqref{C2} and \eqref{C3} (with $\alpha= \beta
= \gamma = p$).  See, for example, Margalit--McCammond\cite{MargalitMcCammond}.

From this we see that the vertex stabiliser is generated by the $p_{ij}$
and $t_k$, that all relations between these elements follow from \eqref{C-pt},
\eqref{C-tt}, \eqref{C1}, \eqref{C2} and \eqref{C3}, and hence the $R_0$
relations are included in $R$.

\section{Edge orbits} \label{edges}

Let $E$ denote the set of all oriented edges that start at $v_0$ the
basepoint of $\X{n}$.  We will now find a representative of each orbit
of the $\FramedPureBraid{n}$ action on $E$, thus giving a set of 
$\PureHilden{2n}$ edge orbit representatives as required by 
Theorem~\ref{methodthm}.  Given an edge $(v_0,v) \in E$, because $v =
\langle D_1, D_2, \ldots, D_n \rangle$ differs from $v_0$ by a simple
move, there exists a unique $i$ such that $D_i \neq d_i$.

If the edge is of length one then there is a unique $d_j$
under $D_i \union d_i$.  All of the remaining discs, $d_k$
for $k \neq i, j$, can be moved by an element of $\FramedPureBraid{n}$
away from $D_i \union d_i$ and then back from behind to their
original positions.  After applying $t_i^p$ for some $p$ we
have one of the following possibilities, each of which lie
in a different orbit.
\begin{longtable}{CCC}
        \edgevovoxij       &         \edgevovoXij         & \quad \for i < j \\[1.2em]
   (v_0, v_0 \cdot x_{ij}) & (v_0, v_0 \cdot x_{ij}^{-1}) & \\[1em]
        \edgevovoyij       &         \edgevovoYij         & \quad \for j < i \\[1.2em]
   (v_0, v_0 \cdot y_{ij}) & (v_0, v_0 \cdot y_{ij}^{-1}) &       
\end{longtable}

Similarly, if the edge is of length two then there exists two discs
$d_j$ and $d_k$, under $d_i \union D_i$.  We may assume that $j < k$. 
As in the previous case there is an element of $\FramedPureBraid{n}$
which takes $(v_0,v)$ to one of the following possibilities, each of
which lie in different orbits.
\begin{longtable}{CCC}
           \edgevovoxijxik             & \edgevovoXikXij & \qquad \for i < j < k \\[1em]
   (v_0, v_0 \cdot x_{ij}\,x_{ik})
                               & (v_0, v_0 \cdot x_{ik}^{-1}\,x_{ij}^{-1})  \\
                                                                        \\
           \edgevovoxjkyij             & \edgevovoYijXjk & \qquad \for j < i < k \\[1em]
   (v_0, v_0 \cdot x_{ik}\,y_{ij})
                               & (v_0, v_0 \cdot y_{ij}^{-1}\,x_{ik}^{-1})  \\
                                                                        \\
           \edgevovoyikyjk             & \edgevovoYjkYik & \qquad \for j < k < i \\[1em]
   (v_0, v_0 \cdot y_{ij}\,y_{ik})
                               & (v_0, v_0 \cdot y_{ik}^{-1}\,y_{ij}^{-1})
\end{longtable}

\begin{proposition}
  The pure Hilden group $\PureHilden{2n}$ is generated by $p_{ij}$,
  $t_i$, $x_{ij}$ and $y_{ij}$.
  \[
     \PureHilden{2n} = \left\langle S \right\rangle
  \]
\end{proposition}

\begin{proof}
  By the Theorem~\ref{methodthm} the group $\PureHilden{2n}$
  is generated by the generators of the vertex stabiliser and 
  $\{ r_\lambda \}$.  We have that
  \[
     \{ r_\lambda \} = \left\{ 
       \begin{array}{c|c}
         \begin{array}{ccc}
           x_{ij},    & x_{ij}^{-1} \\
           y_{ij},    & y_{ij}^{-1}
         \end{array} &  i < j 
       \end{array}
     \right\} \union \left\{
       \begin{array}{c|c}
         \begin{array}{cc}
           x_{ij}\,x_{ik}, & x_{ik}^{-1}\,x_{ij}^{-1} \\
           x_{jk}\,y_{ij}, & y_{ij}^{-1}\,x_{ik}^{-1} \\
           y_{ik}\,y_{jk}, & y_{jk}^{-1}\,y_{ik}^{-1}
         \end{array} & i < j <k
       \end{array}
     \right\}
  \]
  and so all of these generators either are contained in $S$ or can be
  written in terms of the elements of $S$.
\end{proof}

\section{Action of the framed braid group} \label{action-properties}

We have an embedding of the framed braid group on $n$ strings
$\FramedBraid{n}$ in the braid group on $2n$ strings given as
follows.
\[
   \sigma_i = \generatorsigma \qquad \qquad
   \tau_j   = \generatortau
\]
This makes $\FramedBraid{n}$ a subgroup of $\Hilden{2n}$.
It is clear that conjugation by elements of $\FramedBraid{n}$
preserves the pure Hilden group and hence we have a left action
of $\FramedBraid{n}$ on $\PureHilden{2n}$.  In fact this action
can be defined on the level of reduced words as well.  In other
words we have an action of $F\langle \sigma_i,\,\tau_j\rangle$, 
the free group on the letters $\sigma_i$ and $\tau_j$, on 
$F\langle p_{ij},\,x_{ij},\,y_{ij},\,t_k\rangle$, the free
group on the letters $p_{ij}$, $x_{ij}$, $y_{ij}$, $t_k$.
So we have a homomorphism
\begin{eqnarray*}
  F\langle \sigma_i,\,\tau_j \rangle & \longrightarrow 
              & \aut(F\langle p_{ij},\,x_{ij},\,y_{ij},\,t_k\rangle) \\
  g & \longmapsto & \Phi_g
\end{eqnarray*}

In Section~\ref{definePhi} we will construct $\Phi$
and then show that it satisfies the following properties.
For any word $g \in F\langle \sigma_i,\,\tau_j \rangle$,
\begin{enumerate}[(A)]
\item \label{A}
   for each $x \in F\langle p_{ij},\,x_{ij},\,y_{ij},\,t_k\rangle$
   we have $\Phi_g(x) =_{\Braid{2n}} g\,x\,g^{-1}$. 
\item \label{B}
   for any word $h \in F\langle p_{ij},\,t_k\rangle$ we have
   that $\Phi_g(h) \in F\langle p_{ij},\,t_k\rangle$.
\item \label{C}
   for each $r_\lambda$ we have that $\Phi_g(r_\lambda) =_R h_1 r_{\lambda'} h_2$
   for some $h_1, h_2 \in F\langle p_{ij},\,t_k\rangle$ and $r_{\lambda'}$.
\item \label{D}
   if $x =_R y$ then $\Phi_g(x) =_R \Phi_g(y)$.
\end{enumerate}

We will now assume the existance of such a $\Phi$ and use it
to show that $R_1$, $R_2$ and $R_3$ relations follow from 
those in $R$.

\section{The $R_1$ relations} \label{R1}

$R_1$ consist of a relation of the form $r_\lambda\,t r_\lambda^{-1}=h$
for each edge orbit representative $(v_0, v_0 \cdot r_\lambda)$, for
each $t$ in a generating set of the stabiliser of this edge and for some
word $h$ in $\FramedBraid{n}$.

\begin{proposition}
  The stabiliser of the edge $(v_0, v_0 \cdot x_{12})$ is 
  generated as follows.

  \[
     \stab(v_0, v_0\cdot x_{12}) =
        \Bigg\langle
        \begin{array}{cl}
          p_{ij}         & i, j > 2 \\
          t_k            & k > 1    \\
          p_{12} t_1     &          \\
          p_{1k}\ p_{2k} & k > 2  
        \end{array} \Bigg\rangle
  \]

\end{proposition}
\begin{proof}
  As $\stab(v_0, v_0\cdot x_{12})$ is a subgroup of 
  $\stab(v_0) = \FramedPureBraid{n}$ we can view the elements of 
  $\stab(v_0, v_0\cdot x_{12})$ as motions of line segments.
  If we draw a line $L$ between the second and third line 
  segments then this motion can be broken into section consisting
  only of motions of the segments to the right of $L$, sections
  consisting only of motions to the left of $L$ and the motion
  of a single segment across $L$ around both the first and second
  segment and then back across $L$.  The motions to the right
  are generated by $p_{ij}$ for $i,j > 2$ and $t_k$ for $k > 2$.
  The motions to the left are generated by $t_2$ and $p_{12}\ t_1$.
  And the motions across $L$ are of the form $p_{1k}\ p_{2k}$
  for $k > 2$.
\end{proof}

So the $R_1$ relations can be chosen as follows.
\begin{align}
  x_{12}\,\,p_{ij}\,\,x_{12}^{-1}         &= p_{ij} 
               \qquad \for i,j > 2
                                  \label{eq1} \tag{1} \\
  x_{12}\,\,t_k\,\,x_{12}^{-1}            &= t_k    
               \qquad \for k > 1
                                  \label{eq2} \tag{2} \\
  x_{12}\,\,p_{12}\,t_1\,\,x_{12}^{-1}    &= p_{12}\,t_1   
                                  \label{eq3} \tag{3} \\
  x_{12}\,\,p_{1k}\,p_{2k}\,\,x_{12}^{-1} &= p_{1k}\,p_{2k}
               \qquad \for k > 2   
                                  \label{eq4} \tag{4}
\end{align}
Relation \eqref{eq1} follows from \eqref{C1},
relation \eqref{eq2} follows from \eqref{C-xt},
relation \eqref{eq3} follows from \eqref{M-x}
and relation \eqref{eq4} follows from \eqref{C2}.

For the edge orbit representative $(v_0, v_0\cdot x_{12}\,x_{13})$
we can draw a line $L$ between the third and fourth line segment.
Motion of the segments to the right is generated by $p_{ij}$
for $i, j > 3$ and $t_k$ for $k > 3$.  Motion of the segments
to the left is generated by $p_{12}\,p_{13}\,t_1$, $t_2$, $t_3$
and $p_{23}$.  Finally the elements $p_{1k}\,p_{2k}\,p_{3k}$
give the motion between the two halves.  Therefore we have 
the following.

\begin{proposition}
  The stabiliser of the edge $(v_0, v_0\cdot x_{12}\,x_{13})$
  is generated as follows.

  \[
     \stab(v_0, v_0\cdot x_{12}\,x_{13}) =
        \left\langle
        \begin{array}{cl}
          p_{23}                              \\
          p_{ij}                   & i, j > 3 \\
          t_k                      & k > 1    \\
          p_{12}\, p_{13}\, t_1    &          \\  
          p_{1k}\, p_{2k}\, p_{3k} & k > 3  
        \end{array} \right\rangle
  \]

  \qed
\end{proposition}

Hence the $R_1$ relations can be chosen as follows.
\begin{align}
  x_{12}\,x_{13}\,\,p_{23}\,\,(x_{12}\,x_{13})^{-1} 
        &= p_{23}
                          \label{eq5} \tag{5} \\
  x_{12}\,x_{13}\,\,p_{ij}\,\,(x_{12}\,x_{13})^{-1}
        &= p_{ij}                 \qquad \for i,j > 3 
                          \label{eq6} \tag{6} \\
  x_{12}\,x_{13}\,\,t_k\,\,(x_{12}\,x_{13})^{-1}
        &= t_k                    \qquad \for k > 1
                          \label{eq7} \tag{7} \\
  x_{12}\,x_{13}\,\,p_{12}\,p_{13}\,t_1\,\,(x_{12}\,x_{13})^{-1} 
        &= p_{12}\,p_{13}\,t_1 
                          \label{eq8} \tag{8} \\
  x_{12}\,x_{13}\,\,p_{1k}\,p_{2k}\,p_{3k}\,\,(x_{12}\,x_{13})^{-1}
        &= p_{1k}\,p_{2k}\,p_{3k} \qquad \for k > 3
                          \label{eq9} \tag{9}
\end{align}

Relation \eqref{eq5} follows from \eqref{C2},
relation \eqref{eq6} follows from two applications
of \eqref{C1}, relation \eqref{eq7}
follows from two applications of \eqref{C-xt}.
Relation \eqref{eq8} follows from the following.
\begin{longtable}{R@{$\ =\ $}LC} \deductionTableWidth
  \multicolumn{2}{L}{\qquad\qquad x_{12}\,x_{13}\,\underline{p_{12}\,p_{13}}\,t_1} 
                                                    & \eqref{C2} \\
      & x_{12}\,x_{13}\,p_{13}\,p_{23}\,p_{12}\,p_{23}^{-1}\,\underline{t_1} 
                                                    & \eqref{C-pt}^3 \\
      & x_{12}\,\underline{x_{13}\,p_{13}\,t_1}\,p_{23}\,p_{12}\,p_{23}^{-1} 
                                                    & \eqref{M-x} \\
      & x_{12}\,p_{13}\,t_1\,\underline{x_{13}\,p_{23}\,p_{12}}\,p_{23}^{-1}
                                                    & \eqref{C2}   \\
      & x_{12}\,p_{13}\,\underline{t_1}\,p_{23}\,p_{12}\,x_{13}\,p_{23}^{-1}
                                                    & \eqref{C-pt}^2 \\
      & \underline{x_{12}\,p_{13}\,p_{23}}\,p_{12}\,t_1\,x_{13}\,p_{23}^{-1}
                                                    & \eqref{C2}   \\
      & p_{13}\,p_{23}\,\underline{x_{12}\,p_{12}\,t_1}\,x_{13}\,p_{23}^{-1}
                                                    & \eqref{M-x} \\
      & p_{13}\,p_{23}\,p_{12}\,t_1\,\underline{x_{12}\,x_{13}\,p_{23}^{-1}}
                                                    & \eqref{C2}   \\
      & p_{13}\,p_{23}\,p_{12}\,\underline{t_1\,p_{23}^{-1}}\,x_{12}\,x_{13}
                                                    & \eqref{C-pt} \\
      & \underline{p_{13}\,p_{23}\,p_{12}\,p_{23}^{-1}}\,t_1\,x_{12}\,x_{13}
                                                    & \eqref{C2}   \\
      & p_{12}\,p_{13}\,t_1\,x_{12}\,x_{13}
\end{longtable}
Finally \eqref{eq9} follows from the following.
\begin{longtable}{R@{$\ =\ $}LC} \deductionTableWidth
  \multicolumn{2}{L}{\qquad\qquad \underline{x_{13}\,p_{1k}}\,p_{2k}\,p_{3k}} 
                                                     & \eqref{C2}  \\
      & p_{1k}\,p_{3k}\,x_{13}\,\underline{p_{3k}^{-1}\,p_{2k}\,p_{3k}}
                                                     & \eqref{C2}  \\
      & p_{1k}\,p_{3k}\,\underline{x_{13}\,p_{23}\,p_{2k}\,p_{23}^{-1}}
                                                     & \eqref{C3}  \\
      & p_{1k}\,\underline{p_{3k}\,p_{23}\,p_{2k}\,p_{23}^{-1}}\,x_{13}
                                                     & \eqref{C2}  \\
      & p_{1k}\,p_{2k}\,p_{3k}\,x_{13}
\end{longtable}
\begin{longtable}{R@{$\ =\ $}LC} \deductionTableWidth
  \multicolumn{2}{L}{\qquad\qquad \underline{x_{12}\,p_{1k}\,p_{2k}}\,p_{3k}} 
                                                     & \eqref{C2}  \\
      & p_{1k}\,p_{2k}\,\underline{x_{12}\,p_{3k}}   & \eqref{C1}  \\
      & p_{1k}\,p_{2k}\,p_{3k}\,x_{12}
\end{longtable}

Now consider the edge orbit representative $(v_0, v_0 \cdot
r_{\lambda})$ for $r_\lambda \neq x_{12}$ or $x_{12}\,x_{13}$.
There exists some $g \in \FramedBraid{n}$ such that $(v_0, v_0
\cdot r_1) \cdot g = (v_0, v_0\cdot r_\lambda)$, where 
$r_1 = x_{12}$ or $x_{12}\,x_{13}$.  By property 
\eqref{A} of $\Phi$
\[
   \Phi_{g^{-1}}(r_1) =_{\Braid{2n}} g^{-1} r_1 g
\]
and by property \eqref{C} there exists words 
$h_1, h_2 \in \FramedPureBraid{n}$ and some $r_{\lambda'}$
such that
\begin{equation} \label{R1-1}
  \Phi_{g^{-1}}(r_1) =_R h_1\, r_{\lambda'}\, h_2.
\end{equation}
Combining these we see that $v_0 \cdot r_1\,g = 
v_0\cdot r_{\lambda'}\,h_2$ and hence that $\lambda = \lambda'$
and $h_2 \in \stab(v_0, v_0\cdot r_\lambda)$.

Let $T$ be the choice of generators for $\stab(v_0, v_0\cdot r_1)$
chosen above.  So for all $t \in T$ there exists $h \in
\FramedPureBraid{n}$ such that
\[
   r_1\, t\, r_1^{-1} =_R h.
\]
So by property \eqref{D} we have
\begin{equation}\label{R1-2}
  \Phi_{g^{-1}}(r_1\,t\,r_1^{-1}) =_R \Phi_{g^{-1}}(h).
\end{equation}
Property \eqref{B} implies that $\Phi_{g^{-1}}(t)
\in \FramedPureBraid{n}$ and $\Phi_{g^{-1}}(h) \in \FramedPureBraid{n}$.
Combining \eqref{R1-1} and \eqref{R1-2}
we get
\[
   h_1\, r_\lambda\, h_2\,\Phi_{g^{-1}}(t)\,
                        h_2^{-1}\,r_\lambda^{-1}\,h_1^{-1}
      =_R \Phi_{g^{-1}}(h) 
\]
and so $h_2\,\Phi_{g^{-1}}(t)\,h_2^{-1}\in\stab(v_0,v_0\cdot r_\lambda)$.

\begin{claim}
  The set $\{h_2\,\Phi_{g^{-1}}(t)\,h_2^{-1}\st t \in T\}$
  generates $\stab(v_0,v_0\cdot r_\lambda)$.
\end{claim}
\begin{proof}
  As $h_2 \in \stab(v_0,v_0\cdot r_\lambda)$ the set 
  $\{h_2\,\Phi_{g^{-1}}(t)\,h_2^{-1}\st t \in T\}$ generates
  $\stab(v_0,v_0\cdot r_\lambda)$ if and only if the set
  $\{\Phi_{g^{-1}}(t)\st t \in T\}$ generates
  $\stab(v_0,v_0\cdot r_\lambda)$.  This is equivalent to
  saying that for any $s \in \stab(v_0,v_0\cdot r_\lambda)$
  we can find $t_1, \ldots, t_k \in T$ such that
  $s = \Phi_{g^{-1}}(t_1 \cdots t_k)$, in other words that
  $\Phi_g(s) \in \stab(v_0, v_0\cdot r_1)$.
  Now
  \begin{align*}
    (v_0\cdot r_1) \cdot \Phi_g(s) 
        &= v_0\cdot r_1\, g\, s\, g^{-1}  \\
        &= v_0\cdot r_\lambda\, s\, g^{-1}  \\
        &= v_0\cdot r_\lambda \, g^{-1}  \\
        &= v_0\cdot r_1
  \end{align*}
  Therefore the claim holds.
\end{proof}

So for our $R_1$ relation we can choose the following
\[
   r_\lambda\,\,h_2\,\Phi_{g^{-1}}(t)\,h_2^{-1}\,\,r_\lambda^{-1}
          = h_1^{-1}\,\Phi_{g^{-1}}(h)\,h_1
\]
and hence we can choose our $R_1$ relations so that they all
follow from $R$.

\section{The $R_2$ relations} \label{R2}

The $R_2$ relations consist of a relation of the form
$r_{\lambda'}\,h\,r_\lambda = h'$ for each edge orbit
representative, where the LHS is an h-product for the path
$(v_0, v_0\cdot r_\lambda, v_0)$ and $h' \in \FramedBraid{n}$.
For each edge $(v_0, v_0\cdot r_\lambda)$ the edge
$(v_0, v_0\cdot r_\lambda^{-1})$ is in a different orbit.
Our choice of $r_\lambda$ mean that for all $\lambda$ there
exists $\lambda'$ such that $r_\lambda^{-1} = r_{\lambda'}$.
This means that for all the $R_2$ relations we can choose
$r_\lambda^{-1}\, r_\lambda = 1$, i.e.\ they are all trivial.

\section{The $R_3$ relations} \label{R3}

The $R_3$ relations consist of a relation of the form
$r_{\lambda_k}h_k\,\cdots\,r_{\lambda_1}h_1 = h$ for
each face orbit representative, where the LHS is an h-product
that represents the boundary of the face and $h\in\FramedPureBraid{n}$.
As with the $R_1$ relations, we will calculate the relations
for some specific orbits first then use $\Phi$ for the general
case.

There are three types of faces, triangular, non-nested rectangular and
nested rectangular.  Each triangular face orbit is uniquely determined
by $i<j$ and $k$ where $a_i$ and $a_j$ lie under the edges of the
triangle and the changing discs cut out the $k$th arc.  

Each non-nested rectangular face orbit is uniquely determined by four
parameters $i,j$ and $k<l$ where the changing discs cut out the arcs
$a_k$ and $a_l$, $a_i$ is the unique arc lying under the discs that cut
out $a_k$ and $a_j$ is the unique arc lying under the discs that cut
out $a_l$. 

Each nested rectangular face orbit is uniquely determined by three
parameters $i,j,k$ where the changing discs cut out $a_i$ and $a_j$,
$a_k$ is the unique discs lying under the discs cutting out $a_j$, and
$a_j$ and $a_k$ lie under the discs cutting out $a_i$.

We will start with the triangular face $(v_0, v_0 \cdot x_{12}\,x_{13},
v_0\cdot x_{12}, v_0)$.  An h-product for this path is
$x_{13}^{-1}\ x_{12}^{-1}\ (x_{12}\,x_{13})$.  So the $R_3$
relations is
\[
   x_{13}^{-1}\,\,x_{12}^{-1}\,\,(x_{12}\,x_{13}) = 1
\]
and so it is trivial.

Next consider the non-nested rectangular face
\[(v_0,\ v_0 \cdot x_{12},\ v_0 \cdot x_{34}\, x_{12},\ v_0 \cdot x_{34},\ v_0).\]
An h-product that represents this path is
$x_{34}^{-1}\,\,x_{12}^{-1}\,\,x_{34}\,\,x_{12}$.  So the
$R_3$ relations is
\[
   x_{34}^{-1}\,\,x_{12}^{-1}\,\,x_{34}\,\,x_{12} = 1
\]
which follows from \eqref{C1}.

Now consider the nested rectangular face
\[(v_0,\ v_0\cdot x_{23},\ v_0\cdot x_{12}\,x_{13}\,x_{23},\ v_0\cdot x_{12}\,x_{13},\ v_0).\]
An h-product that represents this path is
\[(x_{12}\,x_{13})^{-1}x_{23}^{-1}\,(x_{12}\,x_{13})\,x_{23}.\]
So the $R_3$ relations is
\[
   (x_{12}\,x_{13})^{-1}\,\,x_{23}^{-1}\,(x_{12}\,x_{13})\,x_{23} = 1
\]
which follows from \eqref{C2}.

Given any other face orbit representative $(v_0 = u_0, u_1, \ldots, u_k=v_0)$
there exists some $g \in \FramedBraid{n}$ such that
\[
   (u_0, u_1, \ldots, u_k) = (v_0, v_1, \ldots, v_k) \cdot g
\]
where $(v_0, v_1, \ldots, v_k)$ is the boundary of one of the
three faces whose $R_3$ relations we calculated above.
Suppose the relation from $(v_0, v_1, \ldots, v_k)$ is the following.
\[
   r_{\lambda_k} h_k \cdots r_{\lambda_1} h_1 = h
\]
By property \eqref{C}, for each $r_{\lambda_i}$
there exists $h_{i1}, h_{i2} \in \FramedPureBraid{n}$ and
$r_{\lambda'_i}$ such that
\[
   \Phi_{g^{-1}}(r_{\lambda_i}) =_R h_{i1}\, r_{\lambda'_i}\, h_{i2}
\]

\begin{claim}
  The following h-product represents the path $(u_0, u_1, \ldots, u_k)$.
  \[
     r_{\lambda'_k}\, h_{k2}\, \Phi_{g^{-1}}(h_k)\, h_{(k-1)1}\,
            \cdots\, r_{\lambda'_1}\, h_{11}\, \Phi_{g^{-1}}(h_1)
  \]
\end{claim}
\begin{proof}
  The $i$th vertex of the path associated to the h-product is given as follows.
  \[ \begin{array}{rl}
    \multicolumn{2}{l}{v_0 \cdot r_{\lambda'_i}\, h_{i2}\, \Phi_{g^{-1}}(h_i)\, h_{(i-1)1}\,
            \cdots\, r_{\lambda'_1}\, h_{11}\, \Phi_{g^{-1}}(h_1)}          \\
    = & v_0 \cdot \Phi_{g^{-1}}(r_{\lambda_i}\, h_i \cdots r_{\lambda_1}\, h_1) \\
    = & v_0 \cdot r_{\lambda_i}\, h_i \cdots r_{\lambda_1}\, h_1\, g \\
    = & v_i \cdot g \\
    = & u_i
  \end{array} \]
\end{proof}

Therefore for our $R_3$ relation we may choose the following
\[
   r_{\lambda'_k}\, h_{k2}\, \Phi_{g^{-1}}(h_k)\, h_{(k-1)1}\,
            \cdots\, r_{\lambda'_1}\, h_{11}\, \Phi_{g^{-1}}(h_1)
      = h_{k1}^{-1}\, \Phi_{g^{-1}}(h)
\]
which follows from $R$ by property \eqref{D}.

\section{Construction and properties of $\Phi$} \label{definePhi}

All that remains to prove Theorem~\ref{main-theorem}
is to construct $\Phi$ and show that it satisfies properties
\eqref{A}--\eqref{D}.

Define $\Phi$, the action of $F\langle \sigma_i,\ \tau_j \rangle$
on $F\langle\, p_{ij},\ x_{ij},\ y_{ij},\ t_k\rangle$, as follows.
For $\alpha \in \{p, x, y\}$
\begin{longtable}{L@{\ =\ }LL}
  \Phi_{\sigma_i}\:\!\!(\alpha_{kl})       & \alpha_{kl}
                         & \for i \ne k-1, k, l-1, l         \\
  \Phi_{\sigma_i}\:\!\!(\alpha_{ij})       & \alpha_{i+1,j} 
                         & \for i+1 < j                      \\
  \Phi_{\sigma_i}\:\!\!(\alpha_{i+1,j}) & p_{i,i+1}\,\alpha_{ij}\,p_{i,i+1}^{-1}
                         & \for i+1 < j                      \\
  \Phi_{\sigma_{j}}\:\!\!(\alpha_{i,j+1})   & p_{j,j+1}\,\alpha_{ij}\,p_{j,j+1}^{-1} 
                         & \for i+1 < j                      \\
  \Phi_{\sigma_j}\:\!\!(\alpha_{ij})       & \alpha_{i,j+1} 
                         & \for i+1 < j                      \\[0.5em]
  \Phi_{\sigma_i}\:\!\!(p_{i,i+1}) & p_{i,i+1}                     \\
  \Phi_{\sigma_i}\:\!\!(x_{i,i+1}) & t_{i+1}^{-1}\,y_{i,i+1}\,t_{i+1} \\
  \Phi_{\sigma_i}\:\!\!(y_{i,i+1}) & x_{i,i+1}                     \\[0.5em]
  \Phi_{\sigma_i}\:\!\!(t_j)          & \begin{cases}
                                    t_j     & \text{if $j \ne i$, $i+1$} \\
                                    t_{j+1} & \text{if $j = i$}          \\
                                    t_{i} & \text{if $j = i+1$}   
                                  \end{cases} \\[2em]
  \Phi_{\tau_i}\:\!\!(p_{kl})  & p_{kl}          \\[0.5em]
  \Phi_{\tau_i}\:\!\!(x_{kl})  & \begin{cases}
                             x_{kl}              & \text{if $i \ne k$} \\
                             x_{kl}^{-1}\,p_{kl} & \text{if $i = k$} 
                           \end{cases} & \for k < l \\[1em]
  \Phi_{\tau_i}\:\!\!(y_{kl}) & \begin{cases}
                             y_{kl}              & \text{if $i \ne l$} \\
                             y_{kl}^{-1}\,p_{kl} & \text{if $i = l$}
                           \end{cases} & \for k < l \\[1em]
  \Phi_{\tau_i}\:\!\!(t_j)    & t_j  
\end{longtable}

\begin{proposition} \label{Phi-inverse}
  The map $\Phi$ is a well defined action of
  $F\langle\, \tau_i,\,\sigma_i\, \rangle$ on 
  $F\langle\, p_{ij},\,t_i,\,x_{ij},\,y_{ij}\rangle$.
\end{proposition}

\begin{proof}
  All that needs to be checked is that $\Phi_{\sigma_i}$ and 
  $\Phi_{\tau_i}$ are invertible.  The inverses are as follows.
  \begin{longtable}{L@{\ =\ }LL}
    \Phi_{\sigma_i^{-1}}\:\!\!(\alpha_{kl})       & \alpha_{kl}
                           & \for i \ne k-1, k, l-1, l         \\
    \Phi_{\sigma_i^{-1}}\:\!\!(\alpha_{ij})       & p_{i,i+1}^{-1}\,\alpha_{i+1,j}\,p_{i,i+1}
                           & \for i+1 < j                      \\
    \Phi_{\sigma_i^{-1}}\:\!\!(\alpha_{i+1,j}) & \alpha_{ij}
                           & \for i+1 < j                      \\
    \Phi_{\sigma_j^{-1}}\:\!\!(\alpha_{i,j+1})   & \alpha_{ij} 
                           & \for i+1 < j                      \\
    \Phi_{\sigma_j^{-1}}\:\!\!(\alpha_{ij})       & p_{j,j+1}^{-1}\,\alpha_{i,j+1}\,p_{j,j+1}
                           & \for i+1 < j                      \\[0.5em]
    \Phi_{\sigma_i^{-1}}\:\!\!(p_{i,i+1}) & p_{i,i+1}                     \\
    \Phi_{\sigma_i^{-1}}\:\!\!(x_{i,i+1}) & y_{i,i+1}                  \\
    \Phi_{\sigma_i^{-1}}\:\!\!(y_{i,i+1}) & t_i\,x_{i,i+1}\,t_i^{-1}\\[0.5em]
    \Phi_{\sigma_i^{-1}}\:\!\!(t_j)          & \begin{cases}
                                      t_j     & \text{if $j \ne i$, $i+1$} \\
                                      t_{j+1} & \text{if $j = i$}          \\
                                      t_{j-1} & \text{if $j = i+1$}   
                                    \end{cases} \\[2em]
    \Phi_{\tau_i^{-1}}\:\!\!(p_{kl})  & p_{kl}          \\[0.5em]
    \Phi_{\tau_i^{-1}}\:\!\!(x_{kl})  & \begin{cases}
                                    x_{kl}              & \text{if $i \ne k$} \\
                                    p_{kl}\,x_{kl}^{-1} & \text{if $i = k$} 
                                  \end{cases} & \for k < l \\[1em]
    \Phi_{\tau_i^{-1}}\:\!\!(y_{kl}) & \begin{cases}
                               y_{kl}              & \text{if $i \ne l$} \\
                               p_{kl}\,y_{kl}^{-1} & \text{if $i = l$}
                             \end{cases} & \for k < l \\[1em]
    \Phi_{\tau_i^{-1}}\:\!\!(t_j)    & t_j  
  \end{longtable}
\end{proof}

We will need the following lemma.
\begin{lemma}\label{inverse-squared}
  For $x \in F\langle\, p_{ij},\,t_i,\,x_{ij},\,y_{ij}\rangle$
  we have
  \begin{align*}
    \Phi_{\sigma_{m}^{-2}} & =_R p_{m,m+1}^{-1}\, x\, p_{m,m+1} \\
    \Phi_{\tau_{m}^{-2}}   & =_R t_m^{-1}\, x\, t_m 
  \end{align*} \qed
\end{lemma}

It is easy to check the $\Phi$ satisfies property \eqref{A},
i.e.\ that for every word $g \in F\langle \sigma_i,\,\tau_j \rangle$
and for each $x \in F\langle\, p_{ij},\,t_i,\,x_{ij},\,y_{ij}\rangle$
we have that $\Phi_g(x) = g\, x\, g^{-1}$ as braids.
It is also clear that $\Phi$ satisfies property \eqref{B}.
That is that for any word $g \in F\langle \sigma_i,\,\tau_j \rangle$
and for any word $h \in F\langle\, p_{ij},\,t_k\rangle$ we have 
$\Phi_g(h) \in F\langle\, p_{ij},\,t_k\rangle$.

\begin{proposition}
  The map $\Phi$ satisfies property \eqref{C}, i.e.\ for any word 
  $g \in F\langle \sigma_i,\,\tau_j \rangle$ and any
  \[ 
    r_\lambda \in\left\{ 
       \begin{array}{c|c}
         \begin{array}{ccc}
           x_{ij},    & x_{ij}^{-1} \\
           y_{ij},    & y_{ij}^{-1}
         \end{array} &  i < j 
       \end{array}
     \right\} \union \left\{
       \begin{array}{c|c}
         \begin{array}{cc}
           x_{ij}\,x_{ik}, & x_{ik}^{-1}\,x_{ij}^{-1} \\
           x_{jk}\,y_{ij}, & y_{ij}^{-1}\,x_{ik}^{-1} \\
           y_{ik}\,y_{jk}, & y_{jk}^{-1}\,y_{ik}^{-1}
         \end{array} & i < j <k
       \end{array}
     \right\}
  \]
  we have a relation $\Phi_g(r_\lambda) = h_1 r_{\lambda'} h_2$
  that can be deduced from the relations in $R$, for some $h_1, h_2
  \in F\langle\, p_{ij},\,t_k\rangle$ and some $r_{\lambda'}$.
\end{proposition}

\begin{proof}
  First note that for each word $h$ in $F\langle\, p_{ij},\,t_k\rangle$,
  by property \eqref{B}, the map $\Phi_g$ takes
  $h$ to another word in $F\langle\, p_{ij},\,t_k\rangle$.
  Therefore it suffices to check $\Phi_g$ for $g = \tau_m$,
  $\sigma_m$, $\tau_m^{-2}$ and $\sigma_m^{-2}$.  By 
  Lemma~\ref{inverse-squared} property \eqref{C} is satisfied for
  $g = \tau_m^{-2}$ and $\sigma_m^{-2}$.

  For $r_\lambda = x_{ij}, x_{ij}^{-1}, y_{ij}, y_{ij}^{-1}$
  this follows immediately from the definition of $\Phi$ given
  above.

  Now consider $\Phi_{\sigma_m}(r_\lambda)$ for $r_\lambda =
  x_{ij}\,x_{ik}$, $x_{jk}\,y_{ij}$ or $y_{ik}\,y_{jk}$.  The
  only cases when $\Phi_{\sigma_m}(r_\lambda) \ne r_\lambda$
  are $m=i-1$, $m=i$ and $j=i+1$, $m=i$ and $j>i+1$, $m=j-1$
  and $i<j-1$, $m=j$ and $k=j+1$, $m=j$ and $k>j+1$, $m=k-1$
  and $j<k-1$, and $m=k$.
  \begin{longtable}{LR@{\ =\ }LC}
  %
      m=i-1
        & \Phi_{\sigma_{i-1}}(x_{ij}\,x_{ik}) 
          & p_{i-1,i}\ x_{i-1,j}\,x_{i-1,k}\ p_{i-1,i}^{-1} 
                                  &                         \\[0.7em]
        & \Phi_{\sigma_{i-1}}(x_{jk}\,y_{ij})
          & \underline{x_{jk}\,p_{i-1,i}}\,y_{i-1,j}\,p_{i-1,i}^{-1} 
                                  & \eqref{C1}   \\
        & & p_{i-1,i}\ x_{jk}\,y_{i-1,j}\ p_{i-1,i}^{-1}
                                  &                         \\[0.7em]
        & \Phi_{\sigma_{i-1}}(y_{ik}\,y_{jk})
          & p_{i-1,i}\,y_{i-1,k}\,\underline{p_{i-1,i}^{-1}\,y_{jk}}
                                  & \eqref{C1}   \\
        & & p_{i-1,i}\ y_{i-1,k}\,y_{jk}\ p_{i-1,i}^{-1}
                                  &                         \\[1em]
  %
  %
      m=i \text{ and } j=i+1
        & \Phi_{\sigma_{i}}(x_{ij}\,x_{ik}) 
          & t_j^{-1}\,y_{ij}\,\underline{t_j\,x_{jk}}
                                  & \eqref{M-x} \\
        & & t_j^{-1}\,\underline{y_{ij}\,p_{jk}^{-1}}\,x_{jk}\,p_{jk}t_j
                                  & \eqref{C2}   \\
        & & t_j^{-1}\,p_{jk}^{-1}\,\underline{p_{ik}^{-1}\,y_{ij}\,p_{ik}\,x_{jk}}\,p_{jk}\,t_j
                                  & \eqref{C2}   \\
        & & t_j^{-1}\,p_{jk}^{-1}\ x_{jk}\,y_{ij}\ p_{jk}\,t_j
                                  &                         \\[0.7em]
        & \Phi_{\sigma_{i}}(x_{jk}\,y_{ij})
          & \underline{p_{ij}\,x_{ik}\,p_{ij}^{-1}}\,x_{ij}
                                  & \eqref{C2}   \\
        & & p_{jk}^{-1}\,\underline{x_{ik}\,p_{jk}\,x_{ij}}
                                  & \eqref{C2}   \\
        & & p_{jk}^{-1}\ x_{ij}\,x_{ik}\ p_{jk}
                                  &                         \\[0.7em]
        & \Phi_{\sigma_{i}}(y_{ik}\,y_{jk})
          & \underline{y_{jk}\,p_{ij}\,y_{ik}\,p_{ij}^{-1}}
                                  & \eqref{C2}   \\
        & & y_{ik}\,y_{jk}        &                         \\[1em]
  %
  %
      m=i \text{ and } j>i+1
        & \Phi_{\sigma_{i}}(x_{ij}\,x_{ik}) 
          & x_{i+1,j}\,x_{i+1,k}
                                  &                         \\[0.7em]
        & \Phi_{\sigma_{i}}(x_{jk}\,y_{ij})
          & x_{jk}\,y_{i+1,j}    &                         \\[0.7em]
        & \Phi_{\sigma_{i}}(y_{ik}\,y_{jk})
          & y_{i+1,k}\,y_{jk}    &                         \\[1em]
  %
  %
      m=j-1 \text{ and } i<j-1
        & \Phi_{\sigma_{j-1}}(x_{ij}\,x_{ik})
          & p_{j-1,j}\,x_{i,j-1}\,\underline{p_{j-1,j}^{-1}\,x_{ik}}
                                  & \eqref{C1}   \\
        & &  p_{j-1,j}\ x_{i,j-1}\,x_{ik}\ p_{j-1,j}^{-1}
                                  &                         \\[0.7em]
        & \Phi_{\sigma_{j-1}}(x_{jk}\,y_{ij})
          & p_{j-1,j}\ x_{j-1,k}\,y_{i,j-1}\ p_{j-1,j}^{-1}
                                  &                         \\[0.7em]
        & \Phi_{\sigma_{j-1}}(y_{ik}\,y_{jk})
          & \underline{y_{ik}\,p_{j-1,j}}\,y_{j-1,k}\,p_{j-1,j}^{-1}
                                  & \eqref{C1}   \\
        & & p_{j-1,j}\ y_{ik}\,y_{j-1,k}\ p_{j-1,j}^{-1}
                                  &                         \\[1em]
  %
  %
      m=j \text{ and } k=j+1
        & \Phi_{\sigma_{j}}(x_{ij}\,x_{ik})
          & \underline{x_{ik}\,p_{jk}\,x_{ij}\,p_{jk}^{-1}}
                                  & \eqref{C2}   \\
        & & x_{ij}\,x_{ik}        &                         \\[0.7em]
        & \Phi_{\sigma_{j}}(x_{jk}\,y_{ij})
          & t_k^{-1}\,\underline{y_{jk}}\,t_k\,y_{ik}
                                  & \eqref{M-y} \\
        & & \underline{t_k^{-1}\,p_{jk}\,t_k}\,y_{jk}\,\underline{t_k^{-1}\,p_{jk}^{-1}\,t_k}\,y_{ik}
                                  & \eqref{C-pt}^2 \\
        & & p_{jk}\,y_{jk}\,\underline{p_{jk}^{-1}\,y_{ik}}
                                  & \eqref{C2}   \\
        & & p_{jk}\,\underline{y_{jk}\,p_{ij}\,y_{ik}\,p_{ij}^{-1}}\,p_{jk}^{-1}
                                  & \eqref{C2}   \\
        & & p_{jk}\ y_{ik}\,y_{jk}\ p_{jk}^{-1}
                                  &                         \\[0.7em]
        & \Phi_{\sigma_{j}}(y_{ik}\,y_{jk})
          & \underline{p_{jk}\,y_{ij}\,p_{jk}^{-1}}\,x_{jk}
                                  & \eqref{C2}   \\
        & & \underline{p_{ik}^{-1}\,y_{ij}\,p_{ik}\,x_{jk}}
                                  & \eqref{C2}   \\
        & & x_{jk}\,y_{ij}        &                         \\[1em]
  %
  %
      m=j \text{ and } k>j+1
        & \Phi_{\sigma_{j}}(x_{ij}\,x_{ik}) 
          & x_{i,j+1}\,x_{ik}    &                         \\[0.7em]
        & \Phi_{\sigma_{j}}(x_{jk}\,y_{ij})
          & x_{j+1,k}\,y_{i,j+1}
                                  &                         \\[0.7em]
        & \Phi_{\sigma_{j}}(y_{ik}\,y_{jk})
          & y_{ik}\,y_{j+1,k}    &                         \\[1em]
  %
  %
      m=k-1 \text{ and } j<k-1
        & \Phi_{\sigma_{k-1}}(x_{ij}\,x_{ik}) 
          & \underline{x_{ij}\,p_{k-1,k}}\,x_{i,k-1}\,p_{k-1,k}^{-1}
                                  & \eqref{C1}   \\
        & & p_{k-1,k}\ x_{ij}\,x_{i,k-1}\ p_{k-1,k}^{-1}
                                  &                         \\[0.7em]
        & \Phi_{\sigma_{k-1}}(x_{jk}\,y_{ij})
          & p_{k-1,k}\,x_{j,k-1}\,\underline{p_{k-1,k}^{-1}\,y_{ij}}
                                  & \eqref{C1}   \\
        & & p_{k-1,k}\ x_{j,k-1}\,y_{ij}\ p_{k-1,k}^{-1}
                                  &                         \\[0.7em]
        & \Phi_{\sigma_{k-1}}(y_{ik}\,y_{jk})
          & p_{k-1,k}\ y_{i,k-1}\,y_{j,k-1}\ p_{k-1,k}^{-1}
                                  &                         \\[1em]
  %
  %
      m=k
        & \Phi_{\sigma_{k}}(x_{ij}\,x_{ik})
          & x_{ij}\,x_{i,k+1}    &                         \\[0.7em]
        & \Phi_{\sigma_{k}}(x_{jk}\,y_{ij})
          & x_{j,k+1}\,y_{ij}    &                         \\[0.7em]
        & \Phi_{\sigma_{k}}(y_{ik}\,y_{jk})
          & y_{i,k+1}\,y_{j,k+1} &
  \end{longtable}

  For $\Phi_{\tau_m}$ we only have three cases where 
  $\Phi_{\tau_m}(r_\lambda) \ne r_\lambda$ these are when
  $m=i$ and $r_\lambda = x_{ij}\,x_{ik}$, $m=j$ and 
  $r_\lambda = x_{jk}\,y_{ij}$, and $m=k$ and 
  $r_\lambda = y_{ik}\,y_{jk}$.
  \begin{longtable}{R@{$\ =\ $}LC} \deductionTableWidth
    \Phi_{\tau_i}(x_{ij}\,x_{ik})
      & x_{ij}^{-1}\,\underline{p_{ij}\,x_{ik}^{-1}}\,p_{ik}
                                    & \eqref{C2} \\
      & \underline{x_{ij}^{-1}\,p_{jk}^{-1}\,x_{ik}^{-1}\,p_{jk}}\,p_{ij}\,p_{ik}
                                    & \eqref{C2} \\
      & x_{ik}^{-1}\,x_{ij}^{-1}\ p_{ij}\,p_{ik}          \\[1em]
    \Phi_{\tau_j}(x_{jk}\,y_{ij})
      & x_{jk}^{-1}\,\underline{p_{jk}\,y_{ij}^{-1}}\,p_{ij}
                                    & \eqref{C2} \\
      & \underline{x_{jk}^{-1}\,p_{ik}^{-1}\,y_{ij}^{-1}\,p_{ik}}\,p_{jk}\,p_{ij}
                                    & \eqref{C2} \\
      & y_{ij}^{-1}\,x_{jk}^{-1}\ p_{jk}\,p_{ij}          \\[1em]
    \Phi_{\tau_k}(y_{ik}\,y_{jk})
      & y_{ik}^{-1}\,\underline{p_{ik}\,y_{jk}^{-1}}\,p_{jk}
                                    & \eqref{C2} \\
      & \underline{y_{ik}^{-1}\,p_{ij}^{-1}\,y_{jk}^{-1}\,p_{ij}}\,p_{ik}\,p_{jk}
                                    & \eqref{C2} \\
      & y_{jk}^{-1}\,y_{ik}^{-1}\ p_{ik}\,p_{jk}
  \end{longtable}
 
  For $r_\lambda = x_{ik}^{-1} x_{ij}^{-1}$, 
  $y_{ij}^{-1} x_{ik}^{-1}$ and $y_{jk}^{-1} y_{ik}^{-1}$
  we have shown that for some $h_1, h_2 \in \FramedPureBraid{n}$
  and some $r_{\lambda'}^{-1}$ we have that 
  $\Phi_g(r_\lambda^{-1}) =_R h_1\ r_{\lambda'}^{-1}\ h_2$.  Hence
  we have $\Phi_g(r_\lambda) =_R h_2^{-1}\ r_{\lambda'}\ h_1^{-1}$.
\end{proof}

\begin{proposition} \label{closed-under-phi}
  The map $\Phi$ satisfies property \eqref{D}.
  In other words, for any word $g \in F\langle \sigma_i,\,\tau_j \rangle$
  and any relation $x =_R y$ we have that $\Phi_g(x) =_R \Phi_g(y)$.
\end{proposition}
\begin{proof}
  As in the proof of property \ref{C}, it suffices to show this for $g$ in a 
  monoidal generating set for $F\langle \sigma_i,\,\tau_j \rangle$.  For
  $g=\sigma_i^{-2}$ and $\tau_j^{-2}$ this follows from Lemma~\ref{inverse-squared},
  so it remains to show it for $g = \sigma_i$ and $\tau_j$.

  For any relation only involving $p_{ij}$'s and $t_k$'s the image
  under $\Phi_g$ will still only involve $p_{ij}$'s and $t_k$'s and
  hence, by Proposition~\ref{vertex-stabiliser}, the new relation will
  follow from those in $R$.

  We will now consider the action of $\Phi_{\sigma_q}$ and $\Phi_{\tau_q}$
  on each of the relations.  For any relation $x =_R y$ we will say that
  the deduction of $\Phi_g(x) = \Phi_g(y)$ is trivial if $\Phi_g(x) = \Phi_g(y)$
  is a relation in $R$ of the same type.

  \newcommand{\relation}[2]{{\flushleft\par\medskip\makebox[\textwidth][l]{#1\hspace{\stretch{1}}$#2$\hspace{\stretch{1}}}\par\smallskip}}

  %
  \relation{\eqref{C-xt}}{x_{ij}\,t_k = t_k\,x_{ij}\qquad k\neq i,\ i<j}

  First consider $\Phi_{\sigma_q}$.  If we start with $q=1$ and increase
  it the first non-trivial case is when $q=i-1$.  The next case is when
  $q=i$ and this is only non-trivial if $j=i+1$.  The next case is when
  $q=j-1$ and $j \neq i+1$.  The remaining values are all trivial.

  When $q=i-1$ we have that $\Phi_{\sigma_q}\:\!\!(t_k) = t_{k'}$ where $k' \ne i-1$.
  \begin{longtable}{R@{$\ =\ $}Lc} \deductionTableWidth
    \Phi_{\sigma_q}\:\!\!(x_{ij}\,t_k)
      & p_{i-1,i}\,x_{i-1,j}\,\underline{p_{i-1,i}^{-1}\,t_{k'}}  & \eqref{C-pt}\\
      & p_{i-1,i}\,\underline{x_{i-1,j}\,t_{k'}}\,p_{i-1,i}^{-1}  & \eqref{C-xt}\\
      & \underline{p_{i-1,i}\,t_{k'}}\,x_{i-1,j}\,p_{i-1,i}^{-1}  & \eqref{C-pt}\\
      & t_{k'}\,p_{i-1,i}\,x_{i-1,j}\,p_{i-1,i}^{-1}              & \\
      & \Phi_{\sigma_q}\:\!\!(t_k\,x_{ij})
  \end{longtable}

  When $q=i$ and $j = i+1$ we have that $\Phi_{\sigma_q}\:\!\!(t_k) = 
  t_{k'}$ where $k' \ne j$.
  \begin{longtable}{R@{$\ =\ $}Lc} \deductionTableWidth
    \Phi_{\sigma_q}\:\!\!(x_{ij}\,t_k)
      & t_j^{-1}\,y_{ij}\,\underline{t_j\,t_{k'}}        & \eqref{C-tt}\\
      & t_j^{-1}\,\underline{y_{ij}\,t_{k'}}\,t_j        & \eqref{C-yt}\\
      & \underline{t_j^{-1}\,t_{k'}}\,y_{ij}\,t_j        & \eqref{C-tt}\\
      & t_{k'}\,t_j^{-1}\,y_{ij}\,t_j                    & \\
      & \Phi_{\sigma_q}\:\!\!(t_k\,x_{ij})
  \end{longtable}

  When $q=j-1$ and $j \ne i+1$ we have that $\Phi_{\sigma_q}\:\!\!(t_k) = t_{k'}$
  where $k' \ne i$.
  \begin{longtable}{R@{$\ =\ $}Lc} \deductionTableWidth
    \Phi_{\sigma_q}\:\!\!(x_{ij}\,t_k)
      & p_{j-1,j}\,x_{i,j-1}\,\underline{p_{j-1,j}^{-1}\,t_{k'}}  & \eqref{C-pt}\\
      & p_{j-1,j}\,\underline{x_{i,j-1}\,t_{k'}}\,p_{j-1,j}^{-1}  & \eqref{C-xt}\\
      & \underline{p_{j-1,j}\,t_{k'}}\,x_{i,j-1}\,p_{j-1,j}^{-1}  & \eqref{C-pt}\\
      & t_{k'}\,p_{j-1,j}\,x_{i,j-1}\,p_{j-1,j}^{-1}  & \\
      & \Phi_{\sigma_q}\:\!\!(t_k\,x_{ij})
  \end{longtable}

  Now consider $\Phi_{\tau_q}$, the only non-trivial case is when $q=i$.
  \begin{longtable}{R@{$\ =\ $}Lc} \deductionTableWidth
    \Phi_{\tau_q}\:\!\!(x_{ij}\,t_k)
      & x_{ij}^{-1}\,\underline{p_{ij}\,t_k}  & \eqref{C-pt}\\
      & \underline{x_{ij}^{-1}\,t_k}\,p_{ij}  & \eqref{C-xt}\\
      & t_k\,x_{ij}^{-1}\,p_{ij}              & \\
      & \Phi_{\tau_q}\:\!\!(t_k\,x_{ij})
  \end{longtable}

  %
  \relation{\eqref{C-yt}}{y_{ij}\,t_k = t_k\,y_{ij}\qquad k\neq j,\ i<j}
  First consider $\Phi_{\sigma_q}$, the non-trivial cases are $q=i-1$,
  $q=i$ and $j=i+1$, and $q=j-1$ and $j \neq i+1$.
  
  When $q=i-1$ we have that $\Phi_{\sigma_q}\:\!\!(t_k) = t_{k'}$ where $k' \ne j$.
  \begin{longtable}{R@{$\ =\ $}Lc} \deductionTableWidth
    \Phi_{\sigma_q}\:\!\!(y_{ij}\,t_k)
      & p_{i-1,i}\,y_{i-1,j}\,\underline{p_{i-1,i}^{-1}\,t_{k'}}  & \eqref{C-pt}\\
      & p_{i-1,i}\,\underline{y_{i-1,j}\,t_{k'}}\,p_{i-1,i}^{-1}  & \eqref{C-xt}\\
      & \underline{p_{i-1,i}\,t_{k'}}\,y_{i-1,j}\,p_{i-1,i}^{-1}  & \eqref{C-pt}\\
      & t_{k'}\,p_{i-1,i}\,y_{i-1,j}\,p_{i-1,i}^{-1}  &  \\
      & \Phi_{\sigma_q}\:\!\!(t_k\,y_{ij})
  \end{longtable}
  
  When $q=i$ and $j=i+1$ we have that $\Phi_{\sigma_q}\:\!\!(t_k) = t_{k'}$ where
  $k' \ne i$.
  \begin{longtable}{R@{$\ =\ $}Lc} \deductionTableWidth
    \Phi_{\sigma_q}\:\!\!(y_{ij}\,t_k)
      & \underline{x_{ij}\,t_{k'}}        & \eqref{C-xt}\\
      & t_{k'}\,x_{ij}                    &  \\
      & \Phi_{\sigma_q}\:\!\!(t_k\,y_{ij})
  \end{longtable}
  
  When $q=j-1$ and $j \neq i+1$ we have that $\Phi_{\sigma_q}\:\!\!(t_k) = t_{k'}$
  where $k' \ne j-1$.
  \begin{longtable}{R@{$\ =\ $}Lc} \deductionTableWidth
    \Phi_{\sigma_q}\:\!\!(y_{ij}\,t_k)
        & p_{j-1,j}\,y_{i,j-1}\,\underline{p_{j-1,j}^{-1}\,t_{k'}}  & \eqref{C-pt}\\
        & p_{j-1,j}\,\underline{y_{i,j-1}\,t_{k'}}\,p_{j-1,j}^{-1}  & \eqref{C-yt}\\
        & \underline{p_{j-1,j}\,t_{k'}}\,y_{i,j-1}\,p_{j-1,j}^{-1}  & \eqref{C-pt}\\
        & t_{k'}\,p_{j-1,j}\,y_{i,j-1}\,p_{j-1,j}^{-1}  &  \\
        & \Phi_{\sigma_q}\:\!\!(t_k\,y_{ij})
  \end{longtable}

  Now consider $\Phi_{\tau_q}$, the only non-trivial case is when $q=j$.
  \begin{longtable}{R@{$\ =\ $}Lc} \deductionTableWidth
    \Phi_{\tau_q}\:\!\!(y_{ij}\,t_k)
      & y_{ij}^{-1}\,\underline{p_{ij}\,t_k}  & \eqref{C-pt}\\
      & \underline{y_{ij}^{-1}\,t_k}\,p_{ij}  & \eqref{C-yt}\\
      & t_k\,y_{ij}^{-1}\,p_{ij}              & \\
      & \Phi_{\tau_q}\:\!\!(t_k\,y_{ij})
  \end{longtable}
  
  %
  \relation{\eqref{C1}}{\alpha_{ij}\,\beta_{kl} = \beta_{kl}\,\alpha_{ij}
                       \qquad (i,j,k,l)\text{ cyclically ordered}}
  First consider $\Phi_{\sigma_q}$.  The non-trivial cases are $q=i-1$
  and $i \ne l+1$, $q=i$ and $j=i+1$, $q=j-1$ and $j \neq i+1$, $q=j$
  and $k=j+1$, $q=k-1$ and $j \neq k-1$, $q=k$ and $l = k+1$, $p=l-1$
  and $l \neq k+1$, and $p=l$ and $i=l+1$.

  When $q=i-1$ and $i \ne l+1$ we have the following.
  \begin{longtable}{R@{$\ =\ $}Lc} \deductionTableWidth
    \Phi_{\sigma_q}\:\!\!(\alpha_{ij}\,\beta_{kl})
      & p_{i-1,i}\,\alpha_{i-1,j}\,\underline{p_{i-1,i}^{-1}\,\beta_{kl}}  & \eqref{C1}\\
      & p_{i-1,i}\,\underline{\alpha_{i-1,j}\,\beta_{kl}}\,p_{i-1,i}^{-1}  & \eqref{C1}\\
      & \underline{p_{i-1,i}\,\beta_{kl}}\,\alpha_{i-1,j}\,p_{i-1,i}^{-1}  & \eqref{C1}\\
      & \beta_{kl}\,p_{i-1,i}\,\alpha_{i-1,j}\,p_{i-1,i}^{-1}  &  \\
      & \Phi_{\sigma_q}\:\!\!(\beta_{kl}\,\alpha_{ij})
  \end{longtable}
  
  When $q=i$ and $j=i+1$ the only non-trivial case is when $\alpha=x$.
  \begin{longtable}{R@{$\ =\ $}Lc} \deductionTableWidth
    \Phi_{\sigma_q}\:\!\!(x_{ij}\,\beta_{kl})
      & t_j^{-1}\,y_{ij}\,\underline{t_j\,\beta_{kl}}  & (C-$\beta t$)  \\
      & t_j^{-1}\,\underline{y_{ij}\,\beta_{kl}}\,t_j  & \eqref{C1} \\
      & \underline{t_j^{-1}\,\beta_{kl}}\,y_{ij}\,t_j  & (C-$\beta t$)  \\
      & \beta_{kl}\,t_j^{-1}\,y_{ij}\,t_j  &  \\
      & \Phi_{\sigma_q}\:\!\!(\beta_{kl}\,x_{ij})
  \end{longtable}
  
  When $q=j-1$ and $j \neq i+1$ we have the following.
  \begin{longtable}{R@{$\ =\ $}Lc} \deductionTableWidth
    \Phi_{\sigma_q}\:\!\!(\alpha_{ij}\,\beta_{kl})
      & p_{j-1,j}\,\alpha_{i,j-1}\,\underline{p_{j-1,j}^{-1}\,\beta_{kl}}  & \eqref{C1}\\
      & p_{j-1,j}\,\underline{\alpha_{i,j-1}\,\beta_{kl}}\,p_{j-1,j}^{-1}  & \eqref{C1}\\
      & \underline{p_{j-1,j}\,\beta_{kl}}\,\alpha_{i,j-1}\,p_{j-1,j}^{-1}  & \eqref{C1}\\
      & \beta_{kl}\,p_{j-1,j}\,\alpha_{i,j-1}\,p_{j-1,j}^{-1}  &  \\
      & \Phi_{\sigma_q}\:\!\!(\beta_{kl}\,\alpha_{ij})
  \end{longtable}
  
  When $q=j$ and $k = j+1$ we have the following.
  \begin{longtable}{R@{$\ =\ $}Lc} \deductionTableWidth
    \Phi_{\sigma_q}\:\!\!(\alpha_{ij}\,\beta_{kl})
      & \underline{\alpha_{ik}\,p_{jk}\,\beta_{jl}\,p_{jk}^{-1}}  & \eqref{C3}\\
      & p_{jk}\,\beta_{jl}\,p_{jk}^{-1}\,\alpha_{ik}  &  \\
      & \Phi_{\sigma_q}\:\!\!(\beta_{kl}\,\alpha_{ij})
  \end{longtable}
  
  When $q=k-1$ and $j \neq k-1$ we have the following.
  \begin{longtable}{R@{$\ =\ $}Lc} \deductionTableWidth
    \Phi_{\sigma_q}\:\!\!(\alpha_{ij}\,\beta_{kl})
      & \underline{\alpha_{ij}\,p_{k-1,k}}\,\beta_{k-1,l}\,p_{k-1,k}^{-1}  & \eqref{C1}\\
      & p_{k-1,k}\,\underline{\alpha_{ij}\,\beta_{k-1,l}}\,p_{k-1,k}^{-1}  & \eqref{C1}\\
      & p_{k-1,k}\,\beta_{k-1,l}\,\underline{\alpha_{ij}\,p_{k-1,k}^{-1}}  & \eqref{C1}\\
      & p_{k-1,k}\,\beta_{k-1,l}\,p_{k-1,k}^{-1}\,\alpha_{ij}  &  \\
      & \Phi_{\sigma_q}\:\!\!(\beta_{kl}\,\alpha_{ij})
  \end{longtable}
  
  When $q=k$ and $l=k+1$ the only non-trivial case is when $\beta=x$.
  \begin{longtable}{R@{$\ =\ $}Lc} \deductionTableWidth
    \Phi_{\sigma_q}\:\!\!(\alpha_{ij}\,x_{kl})
      & \underline{\alpha_{ij}\,t_l^{-1}}\,y_{kl}\,t_l  & (C-$\alpha t$)  \\
      & t_l^{-1}\,\underline{\alpha_{ij}\,y_{kl}}\,t_l  & \eqref{C1} \\
      & t_l^{-1}\,y_{kl}\,\underline{\alpha_{ij}\,t_l}  & (C-$\alpha t$)  \\
      & t_l^{-1}\,y_{kl}\,t_l\,\alpha_{ij}  &  \\
      & \Phi_{\sigma_q}\:\!\!(x_{kl}\,\alpha_{ij})
  \end{longtable}
  
  When $q=l-1$ and $l \neq k+1$ we have the following.
  \begin{longtable}{R@{$\ =\ $}Lc} \deductionTableWidth
    \Phi_{\sigma_q}\:\!\!(\alpha_{ij}\,\beta_{kl})
      & \underline{\alpha_{ij}\,p_{l-1,l}}\,\beta_{k,l-1}\,p_{l-1,l}^{-1}  & \eqref{C1}\\
      & p_{l-1,l}\,\underline{\alpha_{ij}\,\beta_{k,l-1}}\,p_{l-1,l}^{-1}  & \eqref{C1}\\
      & p_{l-1,l}\,\beta_{k,l-1}\,\underline{\alpha_{ij}\,p_{l-1,l}^{-1}}  & \eqref{C1}\\
      & p_{l-1,l}\,\beta_{k,l-1}\,p_{l-1,l}^{-1}\,\alpha_{ij}  &  \\
      & \Phi_{\sigma_q}\:\!\!(\beta_{kl}\,\alpha_{ij})
  \end{longtable}
  
  Finally, when $q=l$ and $i=l+1$ we have the following.
  \begin{longtable}{R@{$\ =\ $}Lc} \deductionTableWidth
    \Phi_{\sigma_q}\:\!\!(\alpha_{ij}\,\beta_{kl})
      & \underline{p_{il}\,\alpha_{jl}\,p_{il}^{-1}\,\beta_{ik}}  & \eqref{C3} \\
      & \beta_{ik}\,p_{il}\,\alpha_{jl}\,p_{il}^{-1}  & \\
      & \Phi_{\sigma_q}\:\!\!(\beta_{kl}\,\alpha_{ij})
  \end{longtable}

  Now consider $\Phi_{\tau_q}$, there are two non-trivial cases. 
  In the first case $\Phi_{\tau_q}\:\!\!(\alpha_{ij}) = \alpha_{ij}^{-1} p_{ij}$
  and we have the following.
  \begin{longtable}{R@{$\ =\ $}Lc} \deductionTableWidth
    \Phi_{\tau_q}\:\!\!(\alpha_{ij}\,\beta_{kl})
      & \alpha_{ij}^{-1}\,\underline{p_{ij}\,\beta_{kl}}  & \eqref{C1}\\
      & \underline{\alpha_{ij}^{-1}\,\beta_{kl}}\,p_{ij}  & \eqref{C1}\\
      & \beta_{kl}\,\alpha_{ij}^{-1}\,p_{ij}              & \\
      & \Phi_{\tau_q}\:\!\!(\beta_{kl}\,\alpha_{ij})
  \end{longtable}

  In the second case $\Phi_{\tau_q}\:\!\!(\beta_{kl}) = \beta_{kl}^{-1} p_{kl}$
  and we have the following.
  \begin{longtable}{R@{$\ =\ $}Lc} \deductionTableWidth
    \Phi_{\tau_q}\:\!\!(\alpha_{ij}\,\beta_{kl})
      & \underline{\alpha_{ij}\,\beta_{kl}^{-1}}\,p_{kl}  & \eqref{C1}\\
      & \beta_{kl}^{-1}\,\underline{\alpha_{ij}\,p_{kl}}  & \eqref{C1}\\
      & \beta_{kl}^{-1}\,p_{kl}\,\alpha_{ij}              & \\
      & \Phi_{\tau_q}\:\!\!(\beta_{kl}\,\alpha_{ij})
  \end{longtable}

  %
  \relation{\eqref{C2}}{
       \alpha_{ij}\,\,\beta_{ik}\,\gamma_{jk}
                         = \beta_{ik}\,\gamma_{jk}\,\,\alpha_{ij}
            \qquad \begin{array}{c}
                    (i, j, k) \text{ cyclically ordered,} \\
                    (\alpha, \beta, \gamma)
                    \text{ as in Table~\ref{C2-table}}
             \end{array}}
  First consider $\Phi_{\sigma_q}$.  The only non-trivial cases are when
  $q=i-1$ and $i \neq k+1$, $q=i$ and $j=i+1$, $q=j-1$ and $j \neq i+1$,
  $q=j$ and $k=j+1$, $q=k-1$ and $k \neq j+1$, and $q=k$ and $i=k+1$.

  %
  When $q=i-1$ and $i \neq k+1$ we have the following.
  \begin{longtable}{R@{$\ =\ $}Lc} \deductionTableWidth
    \Phi_{\sigma_q}\:\!\!(\alpha_{ij}\,\beta_{ik}\,\gamma_{jk})
      & p_{i-1,i}\,\alpha_{i-1,j}\,\beta_{i-1,k}\,\underline{p_{i-1,i}^{-1}\,\gamma_{jk}}
                                                    & \eqref{C1} \\
      & p_{i-1,i}\,\underline{\alpha_{i-1,j}\,\beta_{i-1,k}\,\gamma_{jk}}\,p_{i-1,i}^{-1}
                                                    & \eqref{C2} \\
      & p_{i-1,i}\,\beta_{i-1,k}\,\underline{\gamma_{jk}}\,\alpha_{i-1,j}\,p_{i-1,i}^{-1}
                                                    & \eqref{C1} \\
      & p_{i-1,i}\,\beta_{i-1,k}\,p_{i-1,i}^{-1}\,\gamma_{jk}\,p_{i-1,i}\,\alpha_{i-1,j}\,p_{i-1,i}^{-1}
                                                    & \\
      & \Phi_{\sigma_q}\:\!\!(\beta_{ik}\,\gamma_{jk}\,\alpha_{ij})
  \end{longtable}

  %
  When $q=i$ and $j=i+1$ we have two cases. 
  Except for when $i<j<k$ and $(\alpha,\beta,\gamma)=(x,x,p)$ or 
  $k<i<j$ and $(\alpha,\beta,\gamma)=(x,y,p)$ we have the following
  deduction.  Let $\bar t_j$ and $\bar\alpha_{ij}$ be defined as follows.
  \[  \bar t_j = \begin{cases} t_j & \text{if $\alpha = x$} \\
                               1   & \text{if $\alpha \neq x$} \end{cases}
   \qquad
   \bar \alpha_{ij} = \begin{cases} p_{ij} & \text{if $\alpha = p$} \\
                                    y_{ij} & \text{if $\alpha = x$} \\
                                    x_{ij} & \text{if $\alpha = y$} \end{cases} \]
  So we have that $\Phi_{\sigma_q}(\alpha_{ij}) = \bar t_j^{-1}\,\bar\alpha_{ij}\,
  \bar t_j$.
  \begin{longtable}{R@{$\ =\ $}Lc} \deductionTableWidth
    \Phi_{\sigma_q}\:\!\!(\alpha_{ij}\,\beta_{ik}\,\gamma_{jk})
      & \bar t_j^{-1}\,\bar\alpha_{ij}\,\underline{\bar t_j\,\beta_{jk}\,p_{ij}\,\gamma_{ik}\,p_{ij}^{-1}}
                      & \hspace{-3em} (C-$\beta t$) \eqref{C-pt} (C-$\gamma t$) \eqref{C-pt} \\
      & \bar t_j^{-1}\,\bar\alpha_{ij}\,\underline{\beta_{jk}\,p_{ij}\,\gamma_{ik}\,p_{ij}^{-1}}\,\bar t_j
                                                    & \eqref{C2} \\
      & \bar t_j^{-1}\,\underline{\bar\alpha_{ij}\,\gamma_{ik}\,\beta_{jk}}\,\bar t_j  
                                                    & \eqref{C2} \\
      & \underline{\bar t_j^{-1}\,\gamma_{ik}\,\beta_{jk}}\,\bar\alpha_{ij}\,\bar t_j
                                                    & \eqref{C-pt} \eqref{C-pt} \\
      & \underline{\gamma_{ik}\,\beta_{jk}\,p_{ij}}\,p_{ij}^{-1}\,\bar t_j^{-1}\,\bar\alpha_{ij}\,\bar t_j
                                                    & \eqref{C2} \\
      & \beta_{jk} p_{ij}\,\gamma_{ik}\,p_{ij}^{-1}\,\bar t_j^{-1}\,\bar\alpha_{ij}\,\bar t_j
                                                    &  \\
      & \Phi_{\sigma_q}\:\!\!(\beta_{ik}\,\gamma_{jk}\,\alpha_{ij})
  \end{longtable}

  When $i<j<k$ and $(\alpha,\beta,\gamma)=(x,x,p)$ or
  $k<i<j$ and $(\alpha,\beta,\gamma)=(x,y,p)$ we have the following deduction with
  $\beta = x$ or $y$ respectively.
  \begin{longtable}{R@{$\ =\ $}Lc} \deductionTableWidth
    \Phi_{\sigma_q}\:\!\!(x_{ij}\,\beta_{ik}\,p_{jk})
      & t_j^{-1}\,y_{ij}\,t_j\,\underline{\beta_{jk}\,p_{ij}\,p_{ik}}\,p_{ij}^{-1}
                                                    & \eqref{C2} \\
      & \underline{t_j^{-1}\,y_{ij}\,t_j p_{ij}}\,p_{ik}\,\beta_{jk}\,p_{ij}^{-1}
                                                    & \eqref{M-y} \\
      & p_{ij}\,\underline{y_{ij}\,p_{ik}\,\beta_{jk}}\,p_{ij}^{-1}
                                                    & \eqref{C2} \\
      & \underline{p_{ij}\,p_{ik}\,\beta_{jk}}\,y_{ij}\,p_{ij}^{-1}
                                                    & \eqref{C2} \\
      & \beta_{jk}\,p_{ij}\,p_{ik}\,y_{ij}\,p_{ij}^{-1}
                                                    & \eqref{C-pt} \\
      & \beta_{jk}\,p_{ij}\,p_{ik}\,p_{ij}^{-1}\,t_j^{-1}\,\underline{p_{ij}\,t_j\,y_{ij}}\,p_{ij}^{-1}
                                                    & \eqref{M-y} \\
      & \beta_{jk}\,p_{ij}\,p_{ik}\,p_{ij}^{-1}\,t_j^{-1}\,y_{ij}\,\underline{p_{ij}\,t_j}\,p_{ij}^{-1}
                                                    & \eqref{C-pt} \\
      & \beta_{jk}\,p_{ij}\,p_{ik}\,p_{ij}^{-1}\,t_j^{-1}\,y_{ij}\,t_j
                                                    & \\
      & \Phi_{\sigma_q}\:\!\!(\beta_{ik}\,p_{jk}\,x_{ij})
  \end{longtable}

  %
  When $q=j-1$ and $j \neq i+1$ we have the following.
  \begin{longtable}{R@{$\ =\ $}Lc} \deductionTableWidth
    \Phi_{\sigma_q}\:\!\!(\alpha_{ij}\,\beta_{ik}\,\gamma_{jk})
      & p_{j-1,j}\,\alpha_{i,j-1}\,\underline{p_{j-1,j}^{-1}\,\beta_{ik}\,p_{j-1,j}}\,
             \gamma_{j-1,k}\,p_{j-1,j}^{-1}        & \eqref{C1} \\
      & p_{j-1,j}\,\underline{\alpha_{i,j-1}\,\beta_{ik}\,\gamma_{j-1,k}}\,p_{j-1,j}^{-1}
                                                    & \eqref{C2} \\
      & \underline{p_{j-1,j}\,\beta_{ik}}\,\gamma_{j-1,k}\,\alpha_{i,j-1}\,p_{j-1,j}^{-1}
                                                    & \eqref{C1} \\
      & \beta_{ik}\,p_{j-1,j}\,\gamma_{j-1,k}\,\alpha_{i,j-1}\,p_{j-1,j}^{-1} & \\
      & \Phi_{\sigma_q}\:\!\!(\beta_{ik}\,\gamma_{jk}\,\alpha_{ij})
  \end{longtable}

  %
  When $q=j$ and $k = j+1$ we have two cases. 
  Except for when $i<j<k$ and $(\alpha, \beta, \gamma) = (y,p,x)$
  or $j<k<i$ and $(\alpha, \beta, \gamma) = (x,p,x)$ we have the following. Here
  \[  \bar \gamma_{jk} = \begin{cases} p_{jk} & \text{if $\gamma = p$} \\
                                       y_{jk} & \text{if $\gamma = x$} \\
                                       x_{jk} & \text{if $\gamma = y$} \end{cases} \]
  \begin{longtable}{R@{$\ =\ $}Lc} \deductionTableWidth
    \Phi_{\sigma_q}\:\!\!(\alpha_{ij}\,\beta_{ik}\,\gamma_{jk})
      & \alpha_{ik}\,\underline{p_{jk}\beta_{ij}\,p_{jk}^{-1}}\,\bar\gamma_{jk}
                                                    & \eqref{C2} \\
      & \alpha_{ik}\,p_{ik}^{-1}\,\underline{\beta_{ij}\,p_{ik}\,\bar\gamma_{jk}}
                                                    & \eqref{C2} \\
      & \underline{\alpha_{ik}\,\bar\gamma_{jk}\,\beta_{ij}}
                                                    & \eqref{C2} \\
      & \underline{\bar\gamma_{jk}\,\beta_{ij}}\,\alpha_{ik}
                                                    & \eqref{C2} \\
      & \underline{p_{ik}^{-1}\,\beta_{ij}\,p_{ik}}\,\bar\gamma_{jk}\,\alpha_{ik}
                                                    & \eqref{C2} \\
      & p_{jk}\,\beta_{ij}\,p_{jk}^{-1}\,\bar\gamma_{jk}\,\alpha_{ik}
                                                    & \\
      & \Phi_{\sigma_q}\:\!\!(\beta_{ik}\,\gamma_{jk}\,\alpha_{ij})
  \end{longtable}

   When  $i<j<k$ and $(\alpha, \beta, \gamma) = (y,p,x)$ 
  or when $j<k<i$ and $(\alpha, \beta, \gamma) = (x,p,x)$ we have
  \begin{longtable}{R@{$\ =\ $}Lc} \deductionTableWidth
    \Phi_{\sigma_q}\:\!\!(\alpha_{ij}\,p_{ik}\,x_{jk})
      & \alpha_{ik}\,p_{jk}\,p_{ij}\,\underline{p_{jk}^{-1}\,t_k^{-1}\,y_{jk}\,t_k}
                                                    & \eqref{M-y} \\
      & \underline{\alpha_{ik}\,p_{jk}\,p_{ij}}\,y_{jk}\,p_{jk}^{-1}
                                                    & \eqref{C2} \\
      & p_{jk}\,p_{ij}\,\underline{\alpha_{ik}\,y_{jk}}\,p_{jk}^{-1}
                                                    & \eqref{C2} \\
      & p_{jk}\,p_{ij}\,y_{jk}\,\underline{p_{ij}\,\alpha_{ik}\,p_{ij}^{-1}\,p_{jk}^{-1}}
                                                    & \eqref{C2} \\
      & p_{jk}\,p_{ij}\,\underline{y_{jk}\,p_{jk}^{-1}}\,\alpha_{ik}
                                                    & \eqref{M-y} \\
      & p_{jk}\,p_{ij}\,p_{jk}^{-1}\,t_k^{-1}\,y_{jk}\,t_k\,\alpha_{ik}
                                                    & \\
      & \Phi_{\sigma_q}\:\!\!(p_{ik}\,x_{jk}\,\alpha_{ij})
  \end{longtable}

  %
  When $q=k-1$ and $k \neq j+1$ we have the following.
  \begin{longtable}{R@{$\ =\ $}Lc} \deductionTableWidth
    \Phi_{\sigma_q}\:\!\!(\alpha_{ij}\,\beta_{ik}\,\gamma_{jk})
      & \underline{\alpha_{ij}\,p_{k-1,k}}\,\beta_{i,k-1}\,\gamma_{j,k-1}\,p_{k-1,k}^{-1}
                                                    & \eqref{C1} \\
      & p_{k-1,k}\,\underline{\alpha_{ij}\,\beta_{i,k-1}\,\gamma_{j,k-1}}\,p_{k-1,k}^{-1}
                                                    & \eqref{C2} \\
      & p_{k-1,k}\,\beta_{i,k-1}\,\gamma_{j,k-1}\,\underline{\alpha_{ij}\,p_{k-1,k}^{-1}}
                                                    & \eqref{C1} \\
      & p_{k-1,k}\,\beta_{i,k-1}\,\gamma_{j,k-1}\,p_{k-1,k}^{-1}\,\alpha_{ij}
                                                    &  \\
      & \Phi_{\sigma_q}\:\!\!(\beta_{ik}\,\gamma_{jk}\,\alpha_{ij})
  \end{longtable}

  %
  Finally, when $q=k$ and $i = k+1$ we have the following two cases.
  If $\beta \neq x$ then we have the following. Here
  \[  \bar \beta_{ik} = \begin{cases} p_{jk} & \text{if $\beta = p$} \\
                                      y_{jk} & \text{if $\beta = x$} \end{cases} \]
  \begin{longtable}{R@{$\ =\ $}Lc} \deductionTableWidth
    \Phi_{\sigma_q}\:\!\!(\alpha_{ij}\,\beta_{ik}\,\gamma_{jk})
      & \underline{p_{ik}\,\alpha_{jk}\,p_{ik}^{-1}}\,\bar\beta_{ik}\,\gamma_{ij}
                                                    & \eqref{C2} \\
      & \underline{p_{ij}^{-1}\,\alpha_{jk}\,p_{ij}\,\bar\beta_{ik}}\,\gamma_{ij}
                                                    & \eqref{C2} \\[0.4em]
      & \bar\beta_{ik}\,\underline{\alpha_{jk}\,\gamma_{ij}}
                                                    & \eqref{C2} \\
      & \bar\beta_{ik}\,\gamma_{ij}\,p_{ik}\,\alpha_{jk}\,p_{ik}^{-1} \\
      & \Phi_{\sigma_q}\:\!\!(\beta_{ik}\,\gamma_{jk}\,\alpha_{ij})
  \end{longtable}
  
  And if $\beta = x$ then we have the following.
  \begin{longtable}{R@{$\ =\ $}Lc} \deductionTableWidth
    \Phi_{\sigma_q}\:\!\!(\alpha_{ij}\,x_{ik}\,\gamma_{jk})
      & p_{ik}\,\alpha_{jk}\,\underline{p_{ik}^{-1}\,t_i^{-1}}\,y_{ik}\,t_i\,\gamma_{ij}
                                                    & \eqref{C-pt} \\
      & p_{ik}\,\alpha_{jk}\,\underline{t_i^{-1}\,p_{ik}^{-1}\,y_{ik}}\,t_i\,\gamma_{ij}
                                                    & \eqref{M-y} \\
      & p_{ik}\,\alpha_{jk}\,y_{ik}\,\underline{t_i^{-1}\,p_{ik}^{-1}\,t_i}\,\gamma_{ij}
                                                    & \eqref{C-pt} \\
      & p_{ik}\,\alpha_{jk}\,y_{ik}\,\underline{p_{ik}^{-1}\,\gamma_{ij}}
                                                    & \eqref{C2} \\
      & p_{ik}\,\alpha_{jk}\,\underline{y_{ik}\,p_{jk}\,\gamma_{ij}\,p_{jk}^{-1}}\,p_{ik}^{-1}
                                                    & \eqref{C2} \\
      & p_{ik}\,\underline{\alpha_{jk}\,\gamma_{ij}\,y_{ik}}\,p_{ik}^{-1}
                                                    & \eqref{C2} \\
      & p_{ik}\,\underline{\gamma_{ij}\,y_{ik}}\,\alpha_{jk}\,p_{ik}^{-1}
                                                    & \eqref{C2} \\
      & p_{ik}\,y_{ik}\,\underline{p_{jk}\,\gamma_{ij}\,p_{jk}^{-1}}\,\alpha_{jk}\,p_{ik}^{-1}
                                                    & \eqref{C2} \\
      & p_{ik}\,y_{ik}\,\underline{p_{ik}^{-1}}\,\gamma_{ij}\,p_{ik}\,\alpha_{jk}\,p_{ik}^{-1}
                                                    & \eqref{C-pt} \\
      & p_{ik}\,\underline{y_{ik}\,t_i^{-1}\,p_{ik}^{-1}}\,t_i\,\gamma_{ij}\,p_{ik}\,\alpha_{jk}\,p_{ik}^{-1}
                                                    & \eqref{M-y} \\
      & \underline{p_{ik}\,t_i^{-1}\,p_{ik}^{-1}}\,y_{ik}\,t_i\,\gamma_{ij}\,p_{ik}\,\alpha_{jk}\,p_{ik}^{-1}
                                                    & \eqref{C-pt} \\
      & t_i^{-1}\,y_{ik}\,t_i\,\gamma_{ij}\,p_{ik}\,\alpha_{jk}\,p_{ik}^{-1}  \\
      & \Phi_{\sigma_q}\:\!\!(\beta_{ik}\,\gamma_{jk}\,\alpha_{ij})
  \end{longtable}

  %
  Now consider $\Phi_{\tau_q}$, the non-trivial cases are as follows.
  \begin{longtable}{C|CCCC}
  q = i & i<j<k & (x,p,p) & (x,y,y) & (x,x,p) \\
        & j<k<i & (y,p,p) & (y,x,y) & (y,y,p) \\
        & k<i<j & (x,p,p) & (x,x,x) & (x,y,p) \\ \hline
  q = j & i<j<k & (y,p,p) & (y,y,y) & (y,p,x) \\
        & j<k<i & (x,p,p) & (x,x,y) & (x,p,x) \\
        & k<i<j & (y,p,p) & (y,x,x) & (y,p,y) \\ \hline
  q = k & i<j<k & (p,y,y) & (x,y,y) & (y,y,y) \\
        & j<k<i & (p,x,y) & (x,x,y) & (y,x,y) \\
        & k<i<j & (p,x,x) & (x,x,x) & (y,x,x) \\ 
  \end{longtable}

  %
  For the first two columns of the cases $q=i$ and $q=j$ we have the
  following.
  \begin{longtable}{R@{$\ =\ $}Lc} \deductionTableWidth
    \Phi_{\tau_q}\:\!\!(\alpha_{ij}\,\beta_{ik}\,\gamma_{jk})
      & \alpha_{ij}^{-1}\,\underline{p_{ij}\,\beta_{ik}\,\gamma_{jk}} 
                                            & \eqref{C2} \\
      & \underline{\alpha_{ij}^{-1}\,\beta_{ik}\,\gamma_{jk}}\,p_{ij}
                                            & \eqref{C2} \\
      & \beta_{ik}\,\gamma_{jk}\,\alpha_{ij}^{-1}\,p_{ij} & \\
      & \Phi_{\tau_q}\:\!\!(\beta_{ik}\,\gamma_{jk}\,\alpha_{ij})
  \end{longtable}

  %
  For the third column in the case $q=i$ we have the following.
  \begin{longtable}{R@{$\ =\ $}Lc} \deductionTableWidth
    \Phi_{\tau_q}\:\!\!(\alpha_{ij}\,\beta_{ik}\,\gamma_{jk})
      & \alpha_{ij}^{-1}\,p_{ij}\,\beta_{ik}^{-1}\,\underline{p_{ik}\,\gamma_{jk}}
                                            & \eqref{C2} \\
      & \alpha_{ij}^{-1}\,\underline{p_{ij}\,\beta_{ik}^{-1}\,p_{ij}^{-1}}\,p_{ik}\,\gamma_{jk}\,p_{ij}
                                            & \eqref{C2} \\
      & \underline{\alpha_{ij}^{-1}\,p_{jk}^{-1}\,\beta_{ik}^{-1}\,p_{jk}}\,p_{ik}\,\gamma_{jk}\,p_{ij}
                                            & \eqref{C2} \\
      & \beta_{ik}^{-1}\,\underline{\alpha_{ij}^{-1}\,p_{ik}}\,\gamma_{jk}\,p_{ij}
                                            & \eqref{C2} \\

      & \beta_{ik}^{-1}\,p_{ik}\,\gamma_{jk}\,\alpha_{ij}^{-1}\,p_{ij} \\
      & \Phi_{\tau_q}\:\!\!(\beta_{ik}\,\gamma_{jk}\,\alpha_{ij})
  \end{longtable}

  %
  For the third column in the case $q=j$ we have the following.
  \begin{longtable}{R@{$\ =\ $}Lc} \deductionTableWidth
    \Phi_{\tau_q}\:\!\!(\alpha_{ij}\,\beta_{ik}\,\gamma_{jk})
      & \alpha_{ij}^{-1}\,\underline{p_{ij}\,\beta_{ik}\,\gamma_{jk}^{-1}}\,p_{jk}
                                            & \eqref{C2} \\
      & \alpha_{ij}^{-1}\,\gamma_{jk}^{-1}\,\underline{p_{ij}\,\beta_{ik}\,p_{jk}}
                                            & \eqref{C2} \\
      & \underline{\alpha_{ij}^{-1}\,\gamma_{jk}^{-1}}\,\beta_{ik}\,p_{jk}\,p_{ij}
                                            & \eqref{C2} \\
      & \beta_{ik}\,\gamma_{jk}^{-1}\,\underline{\beta_{ik}^{-1}\,\alpha_{ij}^{-1}\,\beta_{ik}\,p_{jk}}\,p_{ij}
                                            & \eqref{C2} \\
      & \beta_{ik}\,\gamma_{jk}^{-1}\,p_{jk}\,\alpha_{ij}^{-1}\,p_{ij} \\
      & \Phi_{\tau_q}\:\!\!(\beta_{ik}\,\gamma_{jk}\,\alpha_{ij})
  \end{longtable}

  %
  For the case when $q=k$ we have the following.
  \begin{longtable}{R@{$\ =\ $}Lc} \deductionTableWidth
    \Phi_{\tau_q}\:\!\!(\alpha_{ij}\,\beta_{ik}\,\gamma_{jk})
      & \alpha_{ij}\,\beta_{ik}^{-1}\,\underline{p_{ik}\,\gamma_{jk}^{-1}}\,p_{jk}
                                            & \eqref{C2} \\
      & \alpha_{ij}\,\underline{\beta_{ik}^{-1}\,p_{ij}^{-1}\,\gamma_{jk}^{-1}\,p_{ij}}\,p_{ik}\,p_{jk}
                                            & \eqref{C2} \\
      & \underline{\alpha_{ij}\,\gamma_{jk}^{-1}\,\beta_{ik}^{-1}}\,p_{ik}\,p_{jk}
                                            & \eqref{C2} \\
      & \gamma_{jk}^{-1}\,\beta_{ik}^{-1}\,\underline{\alpha_{ij}\,p_{ik}\,p_{jk}}
                                            & \eqref{C2} \\
      & \underline{\gamma_{jk}^{-1}\,\beta_{ik}^{-1}}\,p_{ik}\,p_{jk}\,\alpha_{ij}
                                            & \eqref{C2} \\
      & \beta_{ik}^{-1}\,\underline{p_{ij}^{-1}\,\gamma_{jk}^{-1}\,p_{ij}\,p_{ik}}\,p_{jk}\,\alpha_{ij}
                                            & \eqref{C2} \\
      & \beta_{ik}^{-1}\,p_{ik}\,\gamma_{jk}^{-1}\,p_{jk}\,\alpha_{ij} \\
      & \Phi_{\tau_q}\:\!\!(\beta_{ik}\,\gamma_{jk}\,\alpha_{ij})
  \end{longtable}

  %
  \relation{\eqref{C3}}{
                       \alpha_{ik}\,\,p_{jk}\,\beta_{jl}\,p_{jk}^{-1}
                         = p_{jk}\,\beta_{jl}\,p_{jk}^{-1}\,\,\alpha_{ik}
                       \qquad (i,j,k,l)\text{ cyclically ordered}}
  First consider $\Phi_{\sigma_q}$.  As before the only non-trivial
  cases are when $q=i-1$ and $i \neq l+1$, $q=i$ and $j=i+1$, $q=j-1$
  and $j \neq i+1$, $q=j$ and $k=j+1$, $q=k-1$ and $k \neq j+1$, $q=k$
  and $l = k+1$, $p=l-1$ and $l \neq k+1$, and $p=l$ and $i=l+1$.
  
  When $q=i-1$ we have the following.
  \begin{longtable}{R@{$\ =\ $}LC} \deductionTableWidth
    \Phi_{\sigma_q}\:\!\!(\alpha_{ik}\,p_{jk}\,\beta_{jl}\,p_{jk}^{-1})
      & p_{i-1,i}\,\alpha_{i-1,k}\,\underline{p_{i-1,i}^{-1}\,p_{jk}\,\beta_{jl}\,p_{jk}^{-1}} 
                                                    & \eqref{C1} \eqref{C1} \eqref{C1}\\
      & p_{i-1,i}\,\underline{\alpha_{i-1,k}\,p_{jk}\,\beta_{jl}\,p_{jk}^{-1}}\,p_{i-1,i}^{-1}  
                                                    & \eqref{C3}\\
      & \underline{p_{i-1,i}\,p_{jk}\,\beta_{jl}\,p_{jk}^{-1}}\,\alpha_{i-1,k}\,p_{i-1,i}^{-1}  
                                                    & \eqref{C1} \eqref{C1} \eqref{C1}\\
      & p_{jk}\,\beta_{jl}\,p_{jk}^{-1}\,p_{i-1,i}\,\alpha_{i-1,k}\,p_{i-1,i}^{-1}  &  \\
      & \Phi_{\sigma_q}\:\!\!(p_{jk}\,\beta_{jl}\,p_{jk}^{-1}\,\alpha_{ik})
  \end{longtable}
  
  When $q=i$ and $j=i+1$ we have the following.
  (Here the \eqref{C2}s hold because we are in either of the bottom two
  rows of Table~\ref{C2-table}, both of which contain $(\alpha,p,p)$ for
  $\alpha=p$, $x$, and $y$.)
  \begin{longtable}{R@{$\ =\ $}LC} \deductionTableWidth
    \Phi_{\sigma_q}\:\!\!(\alpha_{ik}\,p_{jk}\,\beta_{jl}\,p_{jk}^{-1})
      & \underline{\alpha_{jk}\,p_{ij}\,p_{ik}}\,\beta_{il}\,p_{ik}^{-1}\,p_{ij}^{-1} 
                                                    & \eqref{C2} \\
      & p_{ij}\,p_{ik}\,\underline{\alpha_{jk}\,\beta_{il}}\,p_{ik}^{-1}\,p_{ij}^{-1} 
                                                    & \eqref{C1} \\
      & p_{ij}\,p_{ik}\,\beta_{il}\,\underline{\alpha_{jk}\,p_{ik}^{-1}\,p_{ij}^{-1}} 
                                                    & \eqref{C2} \\
      & p_{ij}\,p_{ik}\,\beta_{il}\,p_{ik}^{-1}\,p_{ij}^{-1}\,\alpha_{jk} & \\
      & \Phi_{\sigma_q}\:\!\!(p_{jk}\,\beta_{jl}\,p_{jk}^{-1}\,\alpha_{ik})
  \end{longtable}
  
  When $q=j-1$ and $j \neq i+1$ we have the following.
  \begin{longtable}{R@{$\ =\ $}LC} \deductionTableWidth
    \Phi_{\sigma_q}\:\!\!(\alpha_{ik}\,p_{jk}\,\beta_{jl}\,p_{jk}^{-1})
      & \underline{\alpha_{ik}\,p_{j-1,j}}\,p_{j-1,k}\,\beta_{j-1,l}\,p_{j-1,k}^{-1}\,p_{j-1,j}^{-1}  
                                                    & \eqref{C1}\\
      & p_{j-1,j}\,\underline{\alpha_{ik}\,p_{j-1,k}\,\beta_{j-1,l}\,p_{j-1,k}^{-1}}\,p_{j-1,j}^{-1}
                                                    & \eqref{C3}\\
      & p_{j-1,j}\,p_{j-1,k}\,\beta_{j-1,l}\,p_{j-1,k}^{-1}\,\underline{\alpha_{ik}\, p_{j-1,j}^{-1}}
                                                    & \eqref{C1}\\
      & p_{j-1,j}\,p_{j-1,k}\,\beta_{j-1,l}\,p_{j-1,k}^{-1}\,p_{j-1,j}^{-1}\, \alpha_{ik}  &  \\
      & \Phi_{\sigma_q}\:\!\!(p_{jk}\,\beta_{jl}\,p_{jk}^{-1}\,\alpha_{ik})
  \end{longtable}
  
  When $q=j$ and $k = j+1$ we have the following.
  \begin{longtable}{R@{$\ =\ $}LC} \deductionTableWidth
    \Phi_{\sigma_q}\:\!\!(\alpha_{ik}\,p_{jk}\,\beta_{jl}\,p_{jk}^{-1})
      & p_{jk}\,\underline{\alpha_{ij}\,\beta_{kl}}\,p_{jk}^{-1} & \eqref{C1} \\
      & p_{jk}\,\beta_{kl}\,p_{jk}^{-1}\,p_{jk}\,\alpha_{ij}\,p_{jk}^{-1} & \\
      & \Phi_{\sigma_q}\:\!\!(p_{jk}\,\beta_{jl}\,p_{jk}^{-1}\,\alpha_{ik})
  \end{longtable}
  
  When $q=k-1$ and $k \neq j+1$ we have the following.
  \begin{longtable}{R@{$\ =\ $}LC} \deductionTableWidth
    \Phi_{\sigma_q}\:\!\!(\alpha_{ik}\,p_{jk}\,\beta_{jl}\,p_{jk}^{-1})
      & p_{k-1,k}\,\alpha_{i,k-1}\,p_{j,k-1}\,\underline{p_{k-1,k}^{-1}\,\beta_{jl}}\,p_{k-1,k}\,p_{j,k-1}^{-1}\,p_{k-1,k}^{-1} & \eqref{C1} \\
      & p_{k-1,k}\,\underline{\alpha_{i,k-1}\,p_{j,k-1}\,\beta_{jl}\,p_{j,k-1}^{-1}}\,p_{k-1,k}^{-1}
                                                    & \eqref{C3} \\
      & p_{k-1,k}\,p_{j,k-1}\,\underline{\beta_{jl}}\,p_{j,k-1}^{-1}\,\alpha_{i,k-1}\,p_{k-1,k}^{-1} 
                                                    & \eqref{C1} \\
      & p_{k-1,k}\,p_{j,k-1}\,p_{k-1,k}^{-1}\,\beta_{jl}\,p_{k-1,k}\,p_{j,k-1}^{-1}\,\alpha_{i,k-1}\,p_{k-1,k}^{-1} & \\
      & \Phi_{\sigma_q}\:\!\!(p_{jk}\,\beta_{jl}\,p_{jk}^{-1}\,\alpha_{ik})
  \end{longtable}
  
  When $q=k$ and $l=k+1$ we have the following. 
  (Here the \eqref{C2}s hold because we are in either of the top two
  rows of Table~\ref{C2-table}, both of which contain $(\beta,p,p)$ for
  $\beta=p$, $x$, and $y$.)
  \begin{longtable}{R@{$\ =\ $}LC} \deductionTableWidth
    \Phi_{\sigma_q}\:\!\!(\alpha_{ik}\,p_{jk}\,\beta_{jl}\,p_{jk}^{-1})
      & \alpha_{il}\,\underline{p_{jl}\,p_{kl}\,\beta_{jk}\,p_{kl}^{-1}\,p_{jl}^{-1}}  
                                                    & \eqref{C2}\\
      & \underline{\alpha_{il}\,\beta_{jk}}       & \eqref{C1}\\
      & \underline{\beta_{jk}}\,\alpha_{il}       & \eqref{C2}\\
      & p_{jl}\,p_{kl}\,\beta_{jk}\,p_{kl}^{-1}\,p_{jl}^{-1}\,\alpha_{il} & \\  
      & \Phi_{\sigma_q}\:\!\!(p_{jk}\,\beta_{jl}\,p_{jk}^{-1}\,\alpha_{ik})
  \end{longtable}
  
  When $q=l-1$ and $l \neq k+1$ we have the following.
  \begin{longtable}{R@{$\ =\ $}LC} \deductionTableWidth
    \Phi_{\sigma_q}\:\!\!(\alpha_{ik}\,p_{jk}\,\beta_{jl}\,p_{jk}^{-1})
      & \underline{\alpha_{ik}\,p_{jk}\,p_{l,l-1}}\,\beta_{j,l-1}\,\underline{p_{l,l-1}^{-1}\,p_{jk}^{-1}}
                                                    & \eqref{C1} \eqref{C1} \eqref{C1} \\
      & p_{l,l-1}\,\underline{\alpha_{ik}\,p_{jk}\,\beta_{j,l-1}\,p_{jk}^{-1}}\,p_{l,l-1}^{-1} 
                                                    & \eqref{C3} \\
      & \underline{p_{l,l-1}\,p_{jk}}\,\beta_{j,l-1}\,\underline{p_{jk}^{-1}\,\alpha_{ik}\,p_{l,l-1}^{-1}}
                                                    & \eqref{C1} \eqref{C1} \eqref{C1} \\
      & p_{jk}\,p_{l,l-1}\,\beta_{j,l-1}\,p_{l,l-1}^{-1}\,p_{jk}^{-1}\,\alpha_{ik} & \\
      & \Phi_{\sigma_q}\:\!\!(p_{jk}\,\beta_{jl}\,p_{jk}^{-1}\,\alpha_{ik})
  \end{longtable}
  
  Finally, when $q=l$ and $i=l+1$ we have the following.
  (Here the \eqref{C2}s hold because they always hold for the triples 
  $(\alpha,p,p)$ and $(\beta,p,p)$.)
  \begin{longtable}{R@{$\ =\ $}LC} \deductionTableWidth
    \Phi_{\sigma_q}\:\!\!(\alpha_{ik}\,\,p_{jk}\,\beta_{jl}\,p_{jk}^{-1})
      & \underline{p_{il}\,\alpha_{kl}\,p_{il}^{-1}}\,\underline{p_{kj}\,\beta_{ij}\,p_{jk}^{-1}}
                                                    & \eqref{C2} \eqref{C2}\\
      & p_{ik}^{-1}\,\underline{\alpha_{kl}\,\beta_{ij}}\,p_{ik}   & \eqref{C1}\\
      & \underline{p_{ik}^{-1}\,\beta_{ij}\,p_{ik}}\,\underline{p_{ik}^{-1}\,\alpha_{kl}\,p_{ik}}
                                                    & \eqref{C2} \eqref{C2} \\
      & p_{jk}\,\beta_{ij}\,p_{jk}^{-1}\,p_{kl}\,\alpha_{kl}\,p_{kl}^{-1} & \\
      & \Phi_{\sigma_q}\:\!\!(p_{jk}\,\beta_{jl}\,p_{jk}^{-1}\,\,\alpha_{ik})
  \end{longtable}

  Now consider $\Phi_{\tau_q}$, there are two non-trivial cases.
  In the first case $\Phi_{\tau_q}\:\!\!(\alpha_{ik}) = \alpha_{ik}^{-1} p_{ik}$
  and we have the following.
  \begin{longtable}{R@{$\ =\ $}Lc} \deductionTableWidth
    \Phi_{\tau_q}\:\!\!(\alpha_{ik}\,\,p_{jk}\,\beta_{jl}\,p_{jk}^{-1})
      & \alpha_{ik}^{-1}\,\underline{p_{ik}\,p_{jk}\,\beta_{jl}\,p_{jk}^{-1}}
                                                    & \eqref{C3}\\
      & \underline{\alpha_{ik}^{-1}\,p_{jk}\,\beta_{jl}\,p_{jk}^{-1}}\,p_{ik}
                                                    & \eqref{C3}\\
      & p_{jk}\,\beta_{jl}\,p_{jk}^{-1}\,\alpha_{ik}^{-1}\,p_{ik}  \\
      & \Phi_{\tau_q}\:\!\!(p_{jk}\,\beta_{jl}\,p_{jk}^{-1}\,\,\alpha_{ik})
  \end{longtable}

  In the second case $\Phi_{\tau_q}\:\!\!(\beta_{jl}) = \beta_{jl}^{-1} p_{jl}$
  and we have the following.
  \begin{longtable}{R@{$\ =\ $}Lc} \deductionTableWidth
    \Phi_{\tau_q}\:\!\!(\alpha_{ik}\,\,p_{jk}\,\beta_{jl}\,p_{jk}^{-1})
      & \underline{\alpha_{ik}\,p_{jk}\,\beta_{jl}^{-1}}\,p_{jl}\,p_{jk}^{-1}
                                                    & \eqref{C3}\\
      & p_{jk}\,\beta_{jl}^{-1}\,\underline{p_{jk}^{-1}\,\alpha_{ik}\,p_{jk}\,p_{jl}\,p_{jk}^{-1}}
                                                    & \eqref{C3}\\
      & p_{jk}\,\beta_{jl}^{-1}\,p_{jl}\,p_{jk}^{-1}\,\alpha_{ik}  \\
      & \Phi_{\tau_q}\:\!\!(p_{jk}\,\beta_{jl}\,p_{jk}^{-1}\,\,\alpha_{ik})
  \end{longtable}

  %
  \relation{\eqref{M-x}}{x_{ij}\,p_{ij}\,t_i = p_{ij}\,t_i\,x_{ij}
                      \qquad i < j}
  First consider $\Phi_{\sigma_q}$.  The only non-trivial cases are when
  $q=i-1$, $q=i$ and $j=i+1$, and $q=j-1$ and $j \neq i+1$.
  
  When $q=i-1$ we have the following.
  \begin{longtable}{R@{$\ =\ $}Lc} \deductionTableWidth
    \Phi_{\sigma_q}\:\!\!(x_{ij}\,p_{ij}\,t_i)
      & p_{i-1,i}\,x_{i-1,j}\,p_{i-1,j}\,\underline{p_{i-1,i}^{-1}\,t_{i-1}}
                                                    & \eqref{C-pt}\\
      & p_{i-1,i}\,\underline{x_{i-1,j}\,p_{i-1,j}\,t_{i-1}}\,p_{i-1,i}^{-1}
                                                    & \eqref{M-x}\\
      & p_{i-1,i}\,p_{i-1,j}\,\underline{t_{i-1}}\,x_{i-1,j}\,p_{i-1,i}^{-1}
                                                    & \eqref{C-pt}\\
      & p_{i-1,i}\,p_{i-1,j}\,p_{i-1,i}^{-1}\,t_{i-1}\,p_{i-1,i}\,x_{i-1,j}\,p_{i-1,i}^{-1} \\
      & \Phi_{\sigma_q}\:\!\!(p_{ij}\,t_i\,x_{ij})
  \end{longtable}
  
  When $q=i$ and $j=i+1$ we have the following.
  \begin{longtable}{R@{$\ =\ $}Lc} \deductionTableWidth
    \Phi_{\sigma_q}\:\!\!(x_{ij}\,p_{ij}\,t_i)
      & t_j^{-1}\,y_{ij}\,\underline{t_j\,p_{ij}}\,t_j & \eqref{C-pt} \\
      & t_j^{-1}\,\underline{y_{ij}\,p_{ij}\,t_j}\,t_j & \eqref{M-y} \\
      & \underline{t_j^{-1}\,p_{ij}\,t_j}\,y_{ij}\,t_j & \eqref{C-pt} \\
      & p_{ij}\,y_{ij}\,t_j   \\
      & \Phi_{\sigma_q}\:\!\!(p_{ij}\,t_i\,x_{ij})
  \end{longtable}
  
  When $q=j-1$ and $j \neq i+1$ we have the following.
  \begin{longtable}{R@{$\ =\ $}Lc} \deductionTableWidth
    \Phi_{\sigma_q}\:\!\!(x_{ij}\,p_{ij}\,t_i)
      & p_{j-1,j}\,x_{i,j-1}\,p_{i,j-1}\,\underline{p_{j-1,j}^{-1}\,t_i}
                                                    & \eqref{C-pt} \\
      & p_{j-1,j}\,\underline{x_{i,j-1}\,p_{i,j-1}\,t_i}\,p_{j-1,j}^{-1}
                                                    & \eqref{M-x} \\
      & p_{j-1,j}\,p_{i,j-1}\,\underline{t_i}\,x_{i,j-1}\,p_{j-1,j}^{-1}
                                                    & \eqref{C-pt} \\
      & p_{j-1,j}\,p_{i,j-1}\,p_{j-1,j}^{-1}\,t_i\,p_{j-1,j}\,x_{i,j-1}\,p_{j-1,j}^{-1}
                                                    & \\
      & \Phi_{\sigma_q}\:\!\!(p_{ij}\,t_i\,x_{ij})
  \end{longtable}
  
  Now consider $\Phi_{\tau_q}$, the only non-trivial case is when $q=i$.
  \begin{longtable}{R@{$\ =\ $}Lc} \deductionTableWidth
    \Phi_{\tau_q}\:\!\!(x_{ij}\,p_{ij}\,t_i)
      & x_{ij}^{-1}\,p_{ij}\,\underline{p_{ij}\,t_i}   & \eqref{C-pt} \\
      & \underline{x_{ij}^{-1}\,p_{ij}\,t_i}\,p_{ij}   & \eqref{M-y} \\
      & p_{ij}\,t_i\,x_{ij}^{-1}\,p_{ij}               & \\
      & \Phi_{\tau_q}\:\!\!(p_{ij}\,t_i\,x_{ij})
  \end{longtable}

  %
  \relation{\eqref{M-y}}{y_{ij}\,p_{ij}\,t_j = p_{ij}\,t_j\,y_{ij}
                      \qquad i < j}
  First consider $\Phi_{\sigma_q}$.  The only non-trivial cases are when
  $q=i-1$, $q=i$ and $j=i+1$, and $q=j-1$ and $j \neq i+1$.
  
  When $q=i-1$ we have the following.
  \begin{longtable}{R@{$\ =\ $}Lc} \deductionTableWidth
    \Phi_{\sigma_q}\:\!\!(y_{ij}\,p_{ij}\,t_j)
      & p_{i-1,i}\,y_{i-1,j}\,p_{i-1,j}\,\underline{p_{i-1,i}^{-1}\,t_j}
                                                    & \eqref{C-pt}\\
      & p_{i-1,i}\,\underline{y_{i-1,j}\,p_{i-1,j}\,t_j}\,p_{i-1,i}^{-1}
                                                    & \eqref{M-y}\\
      & p_{i-1,i}\,p_{i-1,j}\,\underline{t_j}\,y_{i-1,j}\,p_{i-1,i}^{-1}
                                                    & \eqref{C-pt}\\
      & p_{i-1,i}\,p_{i-1,j}\,p_{i-1,i}^{-1}\,t_j\,p_{i-1,i}\,y_{i-1,j}\,p_{i-1,i}^{-1} & \\
      & \Phi_{\sigma_q}\:\!\!(p_{ij}\,t_j\,y_{ij})
  \end{longtable}
  
  When $q=i$ and $j=i+1$ we have the following.
  \begin{longtable}{R@{$\ =\ $}Lc} \deductionTableWidth
    \Phi_{\sigma_q}\:\!\!(y_{ij}\,p_{ij}\,t_j)
      & \underline{x_{ij}\,p_{ij}\,t_i} & \eqref{M-x} \\
      & p_{ij}\,t_i\,x_{ij} &  \\
      & \Phi_{\sigma_q}\:\!\!(p_{ij}\,t_j\,y_{ij})
  \end{longtable}
  
  When $q=j-1$ and $j \neq i+1$ we have the following.
  \begin{longtable}{R@{$\ =\ $}Lc} \deductionTableWidth
    \Phi_{\sigma_q}\:\!\!(y_{ij}\,p_{ij}\,t_j)
      & p_{j-1,j}\,y_{i,j-1}\,p_{i,j-1}\,\underline{p_{j-1,j}^{-1}\,t_{j-1}}
                                                    & \eqref{C-pt} \\
      & p_{j-1,j}\,\underline{y_{i,j-1}\,p_{i,j-1}\,t_{j-1}}\,p_{j-1,j}^{-1}
                                                    & \eqref{M-y} \\
      & p_{j-1,j}\,p_{i,j-1}\,\underline{t_{j-1}}\,y_{i,j-1}\,p_{j-1,j}^{-1}
                                                    & \eqref{C-pt} \\
      & p_{j-1,j}\,p_{i,j-1}\,p_{j-1,j}^{-1}\,t_{j-1}\,p_{j-1,j}\,y_{i,j-1}\,p_{j-1,j}^{-1}
                                                    &  \\
      & \Phi_{\sigma_q}\:\!\!(p_{ij}\,t_j\,y_{ij})
  \end{longtable}
  
  Now consider $\Phi_{\tau_q}$, the only non-trivial case is when $q=j$.
  \begin{longtable}{R@{$\ =\ $}Lc} \deductionTableWidth
    \Phi_{\tau_q}\:\!\!(y_{ij}\,p_{ij}\,t_j)
      & y_{ij}^{-1}\,p_{ij}\,\underline{p_{ij}\,t_j}   & \eqref{C-pt} \\
      & \underline{y_{ij}^{-1}\,p_{ij}\,t_j}\,p_{ij}   & \eqref{M-y} \\
      & p_{ij}\,t_j\,y_{ij}^{-1}\,p_{ij}               & \\
      & \Phi_{\sigma_q}\:\!\!(p_{ij}\,t_j\,y_{ij})
  \end{longtable}
\end{proof}

\bibliographystyle{plain}
\bibliography{pureHilden}

\end{document}